\documentclass[11pt]{article}

\usepackage{authblk}

\usepackage{amsthm,amsmath,amssymb}

\usepackage{graphicx}

\usepackage[a4paper, margin={.15\paperwidth,.1\paperheight},
marginratio=1:1]{geometry}

\RequirePackage[numbers]{natbib}
\RequirePackage[colorlinks,citecolor=blue,urlcolor=blue]{hyperref}

\def\bR{\mathbb R}
\def\bE{\mathbb E}

\newtheorem{Theorem}{Theorem}

\newtheorem{lemma}{Lemma}

\def\argmin{\mathop{\rm arg\,min}}

\def\diag{\mathop{\rm diag}\nolimits}

\def\rank{{\rm rank}}

\def\Sum{\overset{n}{\underset{i=1}{\sum}}}

\newcommand{\norm}[1]{\|#1\|}
\newcommand{\mini}{\mathrm{minimize}}

\newcommand{\R}{\mathbb{R}}

\newcommand{\cP}{\mathcal P}
\newcommand{\cS}{\mathcal S}

\usepackage{algorithm}
\usepackage{algorithmic}

\newcommand{\grad}{\nabla}

\DeclareMathOperator{\col}{col}
\DeclareMathOperator{\prox}{prox}

\numberwithin{equation}{section}
\theoremstyle{plain}

\begin{document}

\title{High dimensional matrix estimation with unknown variance of the noise}

\author[1]{St\'ephane Ga\"iffas}
\author[2]{Olga Klopp}
\affil[1]{\small Centre de Math\'ematiques Appliqu\'ees, \'Ecole Polytechnique \authorcr UMR 7641, 91128 Palaiseau, France}
\affil[2]{\small CREST and University Paris 10, Modal'X, \authorcr CREST 3, Av. Pierre Larousse 92240 Malakoff, France}

\maketitle

\begin{abstract}
Assume that we observe a small set of entries or linear combinations of entries of an unknown matrix $A_0$ corrupted by noise. We propose a new method for estimating $A_0$ which does not rely on the knowledge or on an estimation of the standard deviation of the noise $\sigma$. Our estimator achieves, up to a logarithmic factor, optimal rates of convergence under the Frobenius risk and, thus, has the same prediction performance as previously proposed estimators which rely on the knowledge of $\sigma$. Some numerical experiments show the benefits of this approach. \\


\end{abstract}
%



\section{Introduction}\label{introduction}
In this paper we focus on the problem of high-dimensional matrix estimation from noisy observations with \textit{unknown} variance of the noise. Our main interest is the high dimensional setting, that is, when the dimension of the unknown matrix  is much larger than the sample size. Such problems arise in a variety of applications. In order to obtain a consistent procedure in this setting we need some additional constraints. 
In sparse matrix recovery a standard assumption is that the unknown matrix is exactly or near low-rank. Low-rank conditions are appropriate for many applications such as  recommendation systems, system identification, global positioning, remote sensing (for more details see \cite{candes-plan-noise}).

We propose a new method for approximate low-rank matrix recovery which does not rely on the knowledge or on an estimation of the standard deviation of the noise. Two particular settings are analysed in more details: matrix completion and multivariate linear regression.

In the matrix completion problem we observe a small set of entries of an unknown matrix. Moreover, the entries that we observe may be perturbed by some noise. Based on these observations we want to predict or reconstruct exactly the missing entries. One of the well-known examples of matrix completion is the Netflix recommendation system. Suppose we observe  a few movie ratings from a large data matrix in which rows are users and columns are movies. Each user only watches a few movies compared to the total database of movies available on Netflix. The goal is to predict the missing ratings in order to be able to recommend the movies to a person that he/she has not yet seen. 

In the noiseless setting, if the unknown matrix has low rank and is ``incoherent'', then, it can be reconstructed exactly with high probability from a small set of entries. This result was first
proved by Cand{\`e}s and Recht \cite{candes-recht-exact} using nuclear norm minimization.
A tighter analysis of the same convex relaxation was carried out in \cite{candes-tao-power}. For a simpler approach see  \cite{recht-simpler} and \cite{gross-recovery}. An alternative line of work was developed by Keshavan et al in \cite{keshavan-few}.

In a more realistic setting the observed entries are corrupted by noise. This question has been recently addressed by
several authors  (see, e.g., \cite{candes-plan-noise,keshavan-montanari-matrix,rhode-tsybakov-estimation,2009arxiv0912.5100N,2010arxiv1009.2118N,koltchinskii-von,Koltchinskii-Tsybakov,2010arxiv1008.4886G,klopp-rank}). These results require knowledge of the noise variance, 
however, in practice, such an assumption can be difficult to meet
and the estimation of $\sigma$ is non-trivial in large scale problems. Thus, there is a  gap
between the theory and the practice.  

 The multivariate linear regression model is given by 
\begin{equation}\label{regression}
U_i=V_i\,A_0+E_i\qquad i=1,\dots, l,
\end{equation}
where $U_i$ are $1\times m_2$ vectors of response variables, $V_i$ are $1\times m_1$ vectors of predictors, $A_0$ is an unknown $m_1\times m_2$ matrix of regression coefficients and $E_i$ are random $1\times m_2$ vectors of noise with independent entries and mean zero.
 This model arises in many applications such as the analysis of gene array data, medical imaging, astronomical data analysis, psychometrics and many other areas of applications. 

Previously  multivariate linear regression  with unknown noise variance was considered in \cite{bunea-optimal,giraud-low}. These two papers study rank-penalized estimators. Bunea et al \cite{bunea-optimal}, who first introduced such estimators, proposed an unbiased estimator of $\sigma$ which required an assumption on the dimensions of the problem. This assumption excludes an interesting case, the case when the sample size is smaller than the number of covariates. The method proposed  in \cite{giraud-low} can be applied to this last case under a condition on the rank of the unknown matrix $A_0$. Our method, unlike the method of \cite{bunea-optimal}, can be  applied to the case when the sample size is smaller than the number of covariates and our condition is weaker than the conditions obtained in \cite{giraud-low}. For more details see Section \ref{matrix regression}.

Usually, the variance of the noise is involved in the choice of the regularization parameter. Our main idea is to use the Frobenius norm instead of the squared Frobenius norm as a goodness-of-fit criterion, penalized by the
nuclear norm, which is now a well-established proxy for rank penalization in the
compressed sensing literature \cite{candes-tao-power,gross-recovery}. Roughly, the idea is that in the KKT
condition, the gradient of this “square-rooted” criterion is the regression
score, which is pivotal with respect to the noise level, so that the theoretically
optimal smoothing parameter does not depend on the noise level anymore.

This cute idea for dealing with an unknown noise level was first introduced for square-root lasso by Belloni, Chernozhukov and Wang  \cite{chernozhukov-square} in the vector regression model setting. The estimators proposed in the present paper
require quite a different analysis, with proofs
that differ a lot from the vector case. Other methods dealing with the unknown noise level in high-dimensional sparse regression include e.g. the scaled Lasso \cite{tingni} and the penalized Gaussian log-likelihood \cite{stadler-buhlmann-sara}. For a very complete and comprehensive survey see \cite{giraud-survey}. 
It is an interesting open question if these other methods could be adapted in the matrix setting.


\subsection{Layout of the paper}
This paper is organized as follows. In Section \ref{notation} we set notations. In Section \ref{matrix} we consider the matrix completion problem under uniform sampling at random (USR). We propose a new square-root type estimator for which the choice of the regularization parameter $\lambda$ is independent of $\sigma$. The main result, Theorem \ref{cor1}, shows that, in the case of USR matrix completion and under some mild conditions that link the rank and the ``spikiness'' of $A_0$, the prediction risk of our estimator measured in Frobenius norm is comparable to the sharpest bounds obtained until now.
 
 In Section \ref{matrix regression}, we apply our ideas to the problem of matrix regression. We introduce a new square-root type estimator. For this construction, as in the case of matrix completion, we do not need to know or estimate the noise level. The main result for matrix regression, Theorem \ref{thmr2} gives, up to a logarithmic factor, minimax optimal bound on the prediction error $\left \Vert V\left (\hat A -A_0\right )\right \Vert _2^{2}$.
 
 In Section \ref{simulations} we give empirical results that confirms our theoretical findings.

\subsection{Notation}\label{notation}
For any matrices $A,B\in \mathbb{R}^{m_{1}\times m_{2}}$, we define the scalar product
\begin{equation*}
\langle A,B\rangle =\mathrm{tr}(A^{T}B),
\end{equation*}
where  $\mathrm{tr}(A)$ denotes the trace of the matrix $A$.

For $0<q\leq\infty$ the Schatten-q (quasi-)norm of the matrix $A$  is defined by
\begin{equation*}
\Vert A\Vert_q=\left (\underset{j=1}{\overset{\min(m_1,m_2)}{\sum}}\sigma_j(A)^{q}\right )^{1/q}\text{for}\; 0<q<\infty\;\; \text{and}\; \Vert A\Vert_\infty=\sigma_1(A),
\end{equation*}
where  $(\sigma_j(A))_j$ are the singular values of $A$ ordered decreasingly.


We summarize the notations which we use throughout this paper
\begin{itemize}
\item $\partial G$ is the subdifferential of $G$;
\item $S^\bot$ is the orthogonal complement of $S$; 
\item $\mathcal P_S$ is the orthogonal projector on the linear vector subspace $S$ and $\mathcal P_{S}^{\bot}=1-\mathcal P_{S}$;
\item $\left\Vert A\right\Vert_{\sup}=\underset{i,j}{\max}\mid a_{ij}\mid$ where $A=(a_{ij})$.
\item In what follows we will denote by $c$ a numerical constant whose value can vary from one expression to the other and is independent from $n,m_1,m_2$.
\item Set $m=m_1+m_2$, $m_1\wedge m_2=\min (m_1,m_2)$ and $m_1\vee m_2=\max (m_1,m_2)$.
\item The symbol $\lesssim$ means that the inequality holds up to multiplicative numerical constants.
\end{itemize}


\section{Matrix Completion}\label{matrix}
In this section we construct a square-root estimator for the matrix completion problem under uniform sampling at random. Let $A_0\in \mathbb{R}^{m_{1}\times m_{2}}$ be an unknown matrix, and consider the observations $(X_i,Y_i)$ satisfying the trace regression model
\begin{equation}\label{model}
Y_{i}=\mathrm{tr}(X_{i}^{T}A_{0})+\sigma\xi_{i}, \:i=1,\dots,n.
\end{equation}
Here, $Y_i$ are real random variables; $X_{i}$ are random matrices with dimension $m_{1}\times m_{2}$.
The noise variables $\xi_{i}$ are independent, identically distributed and having distribution $\Phi$ such that
\begin{equation}\label{law}
\bE_\Phi(\xi_i)=0,\:\bE_\Phi(\xi_i^{2})=1
\end{equation}
and  $\sigma>0$ is the \textit{unknown} standard deviation of the noise. 

We assume that the design matrices $X_i$ are i.i.d uniformly distributed on the set
\begin{equation}\label{basisUSR}
\mathcal{X} = \left \{e_j(m_1)e_k^{T}(m_2),1\leq j\leq m_1, 1\leq k\leq m_2\right \},
\end{equation}
where $e_l(m)$ are the canonical basis vectors in $\bR^{m}$. Note that when $X_i=e_j(m_1)e_k^{T}(m_2)$ we observed $(j,k)-$th entry of $A_0$ perturbed by some  noise. 
   When number of observations $n$ is much smaller then the total number of coefficients $m_1\,m_2$, we consider the problem of estimating of $A_0$, i.e. the problem of reconstruction of many missing entries of $A_0$ from $n$ observed coefficients. 
   
In \cite{Koltchinskii-Tsybakov}, the  authors introduce the following estimator of $A_0$ 
\begin{equation}\label{estimator'}
\hat{A}=\underset{A\in \mathbb{R}^{m_{1}\times m_{2}}}{\argmin}\big\{\parallel A-\mathbf{X} \parallel _{2}^{2}+\lambda  \Vert A\Vert_1\big\}
\end{equation}
where
\begin{equation}
\mathbf{X}=\dfrac{m_1\,m_2}{n}\Sum Y_iX_i.
\end{equation}
For this estimator, the variance of the noise is involved in the choice of the regularisation parameter $\lambda$. We propose a new square-root type estimator 
\begin{equation}\label{estimator}
\hat{A}_{\lambda,\mu}=\underset{A\in \mathbb{R}^{m_1\times m_2}}{\argmin}\left \{\Vert A-\mathbf X\Vert_2+\lambda \Vert A\Vert_1\right \}.
\end{equation}
The first part of our estimator coincides with the square root of the data-depending term in \eqref{estimator'}. This is similar to the principle used to define the square-root lasso for the usual vector regression model, see \cite{chernozhukov-square}. Despite taking the square-root of the least squares criterion function, the problem  \ref{estimator} retains global convexity and can be formulated as a solution to a conic programming problem. For more details see Section \ref{simulations}.

We will consider the case of sub-Gaussian noise and matrices with uniformly bounded entries.
Let $a$ denote a constant such that
     \begin{equation}\label{normsup}
     \left\Vert A_0\right\Vert_{\sup}\leq a.
     \end{equation}
 We suppose that the noise variables $\xi_i$ are such that 
\begin{equation}\label{subG1}
\bE(\xi_i)=0,\:\bE(\xi_i^{2})=1
\end{equation}
and there exists a constant $K$ such that 
     \begin{equation}\label{subG}
     \bE\left [\exp(t\xi_i)\right ]\leq \exp\left (t^{2}/2K\right )
     \end{equation}
     for all $t>0$.
     Normal $N(0,1)$ random variables are sub-Gaussian with $K=1$ and \eqref{subG} implies that $\xi_i$ has Gaussian type tails:
     \begin{equation*}
     \mathbb{P}\left \{|\xi_i|>t\right \}\leq 2\exp\left \{-t^{2}/2K\right \}.
     \end{equation*}

Condition $\bE \xi_i^{2}=1$ implies that $K\leq 1$. 

Let us introduce the matrix $\mathbf M$
\begin{equation}\label{M}
\mathbf M=\dfrac{1}{m_1\,m_2}\left (\mathbf X-A_0\right ).
\end{equation}
Note that $\mathbf M$ is centred. Its operator and Frobenius norms play an important role in the choice of the regularisation parameter $\lambda$ (and we will show that they are ``small'' enough). We set
 \begin{equation}\label{delta}
\Delta=\dfrac{\Vert \mathbf M \Vert_{\infty}}{\Vert \mathbf M \Vert_2} .
\end{equation}
The next theorem provides a general oracle inequality for the prediction error of our estimator. Its proof is given in the  Appendix \ref{proof_thm_1}.
\begin{Theorem}\label{thm1}
Suppose that  $\dfrac{\rho}{\sqrt{2\rank(A_0)}}\geq \lambda \geq 3\Delta$ for some $\rho<1$, then
\begin{equation*}
\begin{split}
\Vert \hat A -A_0\Vert _2^{2}&\leq \underset{\sqrt{2\rank (A)}\leq \rho/\lambda}{\mathrm{inf}}
\Big\{\left (1-\rho\right )^{-1}\Vert A-A_0 \Vert_2^{2}
+\left (\dfrac{2\lambda\,m_1\,m_2}{1-\rho}\right )^{2}\Vert  \mathbf M\Vert_2^{2} \;\rank A\Big\}
\end{split}
\end{equation*}
where $\Delta$ and $M$ are defined in \eqref{delta} and \eqref{M}.
\end{Theorem}

In order to specify the value of the regularization parameter $\lambda$, we need to estimate $\Delta$  with high probability. Therefore we use the following two lemmas.
 \begin{lemma}\label{l3} 
For  $n>8(m_1\wedge m_2)\log^{2} m$, with probability at least $1-3/m$,  one has
\begin{equation}\label{tsyb}
\Vert \mathbf M \Vert_{\infty}\leq (c_*\sigma+2a)\sqrt{\dfrac{2\log(m)}{(m_1\wedge m_2)n}}
\end{equation}
where $c_*$ is a numerical 
constant which depends only on $K$. \\If $\xi_i$ are $N(0,1)$, then we can take $c_*=6.5$.
\end{lemma}
\begin{proof}
The bound \eqref{tsyb} is stated in Lemmas 2 and 3 in \cite{Koltchinskii-Tsybakov}. A closer inspection of the proof of Proposition 2 in \cite{koltchinskii-von} gives an estimation on $c_*$ in the case of Gaussian noise. For more details see the Appendix \ref{proof_lemma3}.
\end{proof} 
The following Lemma, proven in the Appendix \ref{proof_Lemma}, provides bounds on $\left\Vert \mathbf M\right\Vert_2$.
\begin{lemma}\label{Lemma}
 Suppose that $4\,n\leq m_1m_2$. Then, for $\mathbf M$ defined in \eqref{M}, with probability at least $1-2/m_1m_2-c_1\exp\{-c_2n\}$, one has 
  \begin{itemize}
  \item[(i)]  
  \begin{equation*}
     2\left  (\dfrac{\left\Vert A_0\right\Vert^2_2}{nm_1m_2} +\dfrac{\sigma^2}{\,n}\right )\geq \left\Vert \mathbf M\right\Vert^{2}_2\geq \dfrac{\sigma^2}{2\,n};
    \end{equation*}
    \item[(ii)]
    \begin{equation*}
    \left\Vert \dfrac{1}{n}\Sum Y_i X_i\right\Vert^{2}_2\geq \dfrac{\left\Vert A_0 \right\Vert_2^{2}}{n\,m_1m_2}\geq \dfrac{4\left\Vert A_0 \right\Vert_2^{2}}{(m_1m_2)^{2}}; 
    \end{equation*}
    \item[(iii)]
        \begin{equation*}
        \left\Vert \mathbf M\right\Vert_2\geq \dfrac{1}{2}\left\Vert \dfrac{1}{n}\Sum Y_i X_i\right\Vert_2
        \end{equation*}
  \end{itemize}
where $(c_1,c_2)$ are numerical constants which depends only on $K,\,a$ and $\sigma$.
\end{lemma}

Recall that the condition on $\lambda$ in Theorem \ref{thm1} is that $\lambda\geq 3\Delta$. Using Lemma \ref{l3} and the lower bounds on $\left\Vert \mathbf{M} \right\Vert_2$ given by Lemma \ref{Lemma}, we can choose 

\begin{equation}\label{lambda}
\lambda=2c_*\sqrt{\dfrac{\log m}{m_1\wedge m_2}}+4a\sqrt{\dfrac{2\,n\log m}{m_1\wedge m_2}}\dfrac{1}{ \left\Vert\Sum Y_iX_i \right\Vert_2}.
\end{equation}
Note that in \eqref{lambda} $\lambda$  is data driving and is independent of $\sigma$. With this choice of $\lambda$, the assumption of Theorem \ref{thm1}, $\dfrac{\rho}{\sqrt{\rank(A_0)}}\geq \lambda$, takes the form
\begin{equation} \label{21}
 \dfrac{\rho}{\sqrt{\rank(A_0)}}\geq 2c_*\sqrt{\dfrac{\log m}{m_1\wedge m_2}}+4a\sqrt{\dfrac{2\,n\log m}{m_1\wedge m_2}}\dfrac{1}{ \left\Vert\Sum Y_iX_i \right\Vert_2}.
\end{equation}
Using (ii) of Lemma \ref{Lemma} we get that \eqref{21} is satisfied with a high probability if
\begin{equation}\label{16}
 \dfrac{\rho}{\sqrt{\rank(A_0)}}\geq 2\,c_*\sqrt{\dfrac{\log m}{m_1\wedge m_2}}+\dfrac{4\,a\sqrt{m_1m_2}}{ \left\Vert A_0\right\Vert_2}\sqrt{\dfrac{2\log m}{m_1\wedge m_2}}.
 \end{equation}
Note that as $m_1$ and $m_2$ are large, the first term in the rhs of \eqref{16} is small. Thus \eqref{16} is essentially equivalent to \begin{equation}\label{20}
 \rho\geq 4\sqrt{\dfrac{2\log m}{(m_1\wedge m_2)}}\sqrt{\rank(A_0)}\,\alpha_{sp}
 \end{equation}
 where $\alpha_{sp}=\dfrac{ \sqrt{m_1m_2}\left\Vert A_0\right\Vert_{\sup}}{\left\Vert A_0\right\Vert_2}$ is the \textit{spikiness ratio} of $A_0$. The notion of ``spikiness'' was introduced by Negahban and Wainwright in \cite{2010arxiv1009.2118N}. We have  that $1\leq \alpha_{sp}\leq \sqrt{m_1m_2}$ and it is large for ``spiky'' matrices, i.e. matrices where some ``large'' coefficients emerge as spikes among  very ``small'' coefficients.   For instance, $\alpha_{sp}=1$ if all the entries of $A_0$ are equal to some constant and $\alpha_{sp}=\sqrt{m_1m_2}$ if $A_0$ has only one non-zero entry.  
 
 Condition \eqref{20} is a kind of trade-off between ``spikiness'' and rank. If $\alpha_{sp}$ is bounded by a constant, then, up to a logarithmic factor, $\rank(A_0)$ can be of the order $m_1\wedge m_2$, which is its maximal possible value. If our matrix is ``spiky'', then we need low rank. To give some intuition let us consider the case of square matrices. Typically, matrices with both high spikiness ratio and high rank look almost diagonal. Thus, under uniform sampling and if $n\ll m_1m_2$,  with high probability we do not observe diagonal (i.e. non-zero) elements. 
%

  \begin{Theorem}\label{cor1}
  Let the set of conditions \eqref{subG1} - \eqref{normsup} be satisfied and $\lambda$ be as in \eqref{lambda}. Assume that $8(m_1\wedge m_2)\log^{2} m < n\leq \dfrac{m_1m_2}{4}$  and that \eqref{16} holds for some $\rho<1$.  Then, with probability at least $1-4/m-c_1\exp\{-c_2n\}$
  \begin{equation}\label{thm7}
  \dfrac{1}{m_1m_2}\Vert \hat A -A_0\Vert _2^{2}\leq C_*\dfrac{(m_1\vee m_2)}{n}\rank(A_0)\log m.
  \end{equation}
  Here $C_*=\dfrac{16\left (2c_*\sigma^2+(18+2c_*)a^{2}\right )}{(1-\rho)^{2}}$, $c_*$ is an absolute constant that depends only on $K$ and $(c_1,c_2)$ are numerical constants that depend only on $K,\,a$ and $\sigma$.
  \end{Theorem}
  \begin{proof}
  This is a consequence of Theorem \ref{thm1} for $A=A_0$. From \eqref{lambda} we get 
  \begin{equation}\label{18}
  \begin{split}
  \Vert \hat A -A_0\Vert _2^{2}&\leq \dfrac{8(m_1m_2)^{2}}{(1-\rho)^{2}}\left (c_*\sqrt{\dfrac{4\log m}{m_1\wedge m_2}}+2a\sqrt{\dfrac{2\,n\log m}{m_1\wedge m_2}}\dfrac{1}{ \left\Vert\Sum Y_iX_i \right\Vert_2}\right )^{2}\\&\hskip 1 cm\times \left\Vert \mathbf M \right\Vert_2^{2}\rank(A_0).
  \end{split}
  \end{equation}
  Using triangle inequality and (ii) of Lemma \ref{Lemma} we compute 
  \begin{equation*}\label{19}
  \begin{split}
    \left\Vert \mathbf M \right\Vert_2&\leq \left\Vert\dfrac{1}{n}\Sum Y_iX_i \right\Vert_2+\dfrac{1}{m_1m_2} \left\Vert A_0\right\Vert_2\\&\leq \dfrac{3}{2}\left \Vert\dfrac{1}{n}\Sum Y_iX_i \right\Vert_2.
  \end{split}
  \end{equation*}
 Using (i) of Lemma \ref{Lemma} and \eqref{19}, from \eqref{18}, we get 
  \begin{equation*}
    \begin{split}
    \Vert \hat A -A_0\Vert _2^{2}\leq &\dfrac{16\log (m)(m_1m_2)^{2}}{(1-\rho)^{2}(m_1\wedge m_2)}\left (2\,c_*\left  (\dfrac{\left\Vert A_0\right\Vert^2_2}{nm_1m_2} +\dfrac{\sigma^2}{\,n}\right )+\dfrac{18a^{2}}{n}\right ) \rank(A_0).
    \end{split}
    \end{equation*}
    Then, we use $\left\Vert A_0\right\Vert^2_2\leq a^{2}m_1m_2$ to obtain
    \begin{equation*}
        \begin{split}
        \dfrac{\Vert \hat A -A_0\Vert _2^{2}}{m_1m_2}\leq &\dfrac{16\log (m)(m_1\vee m_2)}{(1-\rho)^{2}n}\left (2c_*\sigma^2+(18+2c_*)a^{2}\right ) \rank(A_0).
        \end{split}
        \end{equation*}
        This completes the proof of Theorem \ref{cor1}.
  \end{proof}
  Theorem \ref{cor1} guarantees that the normalized Frobenius error $\dfrac{\parallel \hat{A}-A_{0}\parallel _{2}}{\sqrt{m_1m_2}}$ of the estimator $\hat A$ is small whenever $n>C(m_1\vee m_2)\log(m)\rank (A_0)$ with a constant $C$ large enough. This quantifies the sample size, n, necessary for
  successful matrix completion from noisy data with unknown variance of the noise. Remarkably, this sampling size is the same as in the case of known variance of the noise. In Theorem \ref{cor1} we have an additional restriction $4n\leq m_1m_2$.  In matrix completion setting the number of observed entries $n$ is always smaller then the total number of entries $m_1m_2$ and this condition can be replaced by $n\leq\alpha m_1m_2$ for some $\alpha <1$.
  
Theorem \ref{cor1} leads to the same rate of convergence as previous results on matrix completion which treat $\sigma$ as known.   In order to compare our bounds to those obtained in past works on
  noisy matrix completion, we will start with describing the result of Keshavan et al \cite{keshavan-montanari-matrix}.
  Under a sampling scheme different from ours (sampling without replacement) and sub-Gaussian errors, the estimator proposed in \cite{keshavan-montanari-matrix} satisfies, with  high probability, the following bound
 \begin{equation}\label{keshavan}
   \dfrac{1}{m_1m_2}\Vert \hat A -A_0\Vert _2^{2}\lesssim k^{4}\sqrt{\alpha}\dfrac{(m_1\vee m_2)}{n}\rank(A_0)\log n.
   \end{equation}
   Here $k=\sigma_{\max}(A_0)/\sigma_{\min}(A_0)$ is the condition number and $\alpha=(m_1\vee m_2)/(m_1\wedge m_2)$ is the aspect ratio. Comparing \eqref{keshavan} and \eqref{thm7}, we see that our bound is better: it does not involve the multiplicative coefficient  $k^{4}\sqrt{\alpha}$ which can be big.
   
   Negahban et al in \cite{2010arxiv1009.2118N} propose an estimator which, in the case of USR matrix completion and sub-exponential noise, satisfies 
   \begin{equation} \label{wainwright}
      \dfrac{1}{m_1m_2}\Vert \hat A -A_0\Vert _2^{2}\lesssim \alpha_{sp}\dfrac{m}{n}\rank(A_0)\log m.
      \end{equation}
      Here $\alpha_{sp}$ is the spikiness ratio of $A_0$. For $\alpha_{sp}$ bounded by a constant, \eqref{wainwright} gives the same bound as Theorem \ref{cor1}. The
      construction of $\hat A$ in \cite{2010arxiv1009.2118N} requires a priori information on the spikiness ratio of $A_0$ and on $\sigma$. This is not the case for our estimator.
      
      The estimator proposed by Koltchinskii et al in \cite{Koltchinskii-Tsybakov} achieves the same bound as ours. In addition  to prior information on $\left\Vert A_0\right\Vert_{\sup}$, their method also requires prior information on $\sigma$.  In the case of Gaussian errors, this rate of convergence is optimal up to a logarithmic factor (cf. Theorem 6 of \cite{Koltchinskii-Tsybakov}) for the class of matrices $\mathcal{A}(r,a)$ defined as follows: for given $r$ and $a$, $A_0\in \mathcal{A}(r,a)$ if and only if the rank of $A_0$ is bounded by $r$ and all the entries of $A_0$ are  bounded in absolute value by $a$.
      
     One important difference with previous works on matrix completion is that Theorem \ref{cor1} requires the additional growth restriction on $\lambda$, that is  the condition $\dfrac{\rho}{\sqrt{2\rank(A_0)}}\geq \lambda$.  The consequence of this growth restriction is that our method can not be applied to matrices which have both large spikiness ratio and large rank. 
     Note that  the square-root lasso estimator also requires an additional growth restriction on $\lambda$ (see Theorem 1 in \cite{chernozhukov-square}). We may think that these restrictions is the price of not knowing $\sigma$ in our framework.
  \section{Matrix Regression}\label{matrix regression}
     In this section we apply our method to matrix regression.  Recall that the matrix regression model is given by 
     \begin{equation}\label{regression_bis}
     U_i=V_i\,A_0+E_i\qquad i=1,\dots, n,
     \end{equation}
     where $U_i$ are $1\times m_2$ vectors of response variables; $V_i$ are $1\times m_1$ vectors of predictors; $A_0$ is an unknown $m_1\times m_2$ matrix of regression coefficients; $E_i$ are random $1\times m_2$ noise vectors with independent entries $E_{ij}$. We suppose that  $E_{ij}$ has mean zero and \textit{unknown} standard deviation $\sigma$.
        Set 
     $V=\left  (V_1^{T},\dots, V_n^{T}\right  )^{T}$, $U=\left  (U_1^{T},\dots, U_n^{T}\right  )^{T}$ and $E=\left  (E_1^{T},\dots, E_n^{T}\right  )^{T}$. 
     
     
     We propose new estimator of $A_0$ using again the idea of the square-root estimators:
     \begin{equation*}\label{estimator_regression}
     \hat{A}=\underset{A\in \mathbb{R}^{m_{1}\times m_{2}}}{\argmin}\left \{\Vert U-V\,A\Vert_2+\lambda \Vert VA\Vert_1\right \},
     \end{equation*}
     where $\lambda>0$ is a regularization parameter.  This estimator can be formulated as a solution to a conic programming problem. For more details see Section \ref{simulations}.
     
     Recall that $\mathcal P_{V}$ denote the orthogonal projector on the linear span of the columns of matrix $V$.  We set $$\Delta'=\dfrac{ \left\Vert\mathcal{P}_V\left (E\right ) \right\Vert_\infty}{ \left\Vert E\right\Vert_2}.$$
     Minor modifications in the proof of Theorem \ref{thm1} yield the following result.
     \begin{Theorem}\label{thmr1} 
     Suppose that  $\dfrac{\rho}{\sqrt{2\rank(VA_0)}}\geq \lambda \geq 3\Delta'$ for some $\rho<1$, then
     \begin{equation*}
     \begin{split}
     \left \Vert V\left (\hat A -A_0\right )\right \Vert _2^{2}&\leq \underset{\sqrt{2\rank (VA)}\leq \rho/\lambda}{\mathrm{inf}}
     \left \{\dfrac{\left \Vert V\left (A-A_0\right ) \right \Vert_2^{2}}{1-\rho}
     +\left (\dfrac{2\lambda}{1-\rho}\right )^{2}\left \Vert   E\right \Vert_2^{2}\, \rank (VA)\right \}
     \end{split}
     \end{equation*}
     \end{Theorem}
     \begin{proof} The proof follows the lines of the proof of Theorem \ref{thm1} and it is given in the Appendix \ref{proof_thmr1}.
     \end{proof}
     
     To get the oracle inequality in a closed form it remains to specify the value of regularization parameter $\lambda$ such that $\lambda\geq 3\Delta'$. This requires some assumptions on the distribution of the noise $(E_{ij})_{i,j}$. We will consider the case of Gaussian errors. Suppose that $E_{ij}=\sigma\xi_{ij}$ where $\xi_{ij} $ are  normal $N(0,1)$ random variables. In order to estimate  $\left\Vert \mathcal{P}_VE\right\Vert_\infty$ we will use the following result proven in \cite{bunea-optimal}.
      \begin  {lemma}[\cite {bunea-optimal}, Lemma 3]\label{lemma_bunea}
              Let $r=\rank(V)$ and assume that $E_{ij}$ are independent $N(0,\sigma^{2})$ random variables. Then
                \begin{equation*}
               \bE (\left \| \mathcal{P}_{V}E\right\|_\infty)\leq \sigma(\sqrt{m_2}+\sqrt{r})
               \end{equation*}
               and\begin{equation*}
               \mathbb P\left \{\left \| \mathcal{P}_{V}E\right\|_\infty\geq \bE (\left \| \mathcal{P}_{V}E\right\|_\infty)+\sigma t\right \}\leq \exp\left \{-t^{2}/2\right \}.
               \end{equation*}
     \end{lemma}
%
     We use Bernstein's inequality to get a bound on  $ \left\Vert  E \right\Vert_2$. Let $\alpha<1$. With  probability at least 
     $1-2\exp\left \{-c\,\alpha^{2}\,n\,m_2\right \}$, one has 
     \begin{equation}\label{r32}
     (1+\alpha)\sigma\sqrt{n\,m_2}\geq\left\Vert  E \right\Vert_2\geq (1-\alpha)\sigma\sqrt{n\,m_2}.
     \end{equation}
     Let $\beta>0$ and take $t=\beta \left (\sqrt{m_2}+\sqrt{r}\right )$ in Lemma \ref{lemma_bunea}. Then, using \eqref{r32},
     we can take \begin{equation}\label{r21}
     \lambda=\dfrac{(1+\beta)\left (\sqrt{m_2}+\sqrt{r}\right )}{(1-\alpha)\sqrt{n\,m_2}}.
     \end{equation}
     Put $\gamma=\dfrac{1+\beta}{1-\alpha}>1$. Thus, condition $\dfrac{\rho}{\sqrt{2\,\rank(VA_0)}}\geq \lambda$  gives
     \begin{equation}\label{r22}
    \rank(VA_0)\leq \dfrac{\rho^{2}n\,m_2}{2\gamma^{2}\left (\sqrt{m_2}+\sqrt{r}\right )^{2}}
     \end{equation}
     and we get the following result.
     \begin{Theorem}\label{thmr2}
     Assume that $\xi_{ij}$ are independent $N(0,1)$. Pick $\lambda$ as in \eqref{r21}. Assume \eqref{r22} is satisfied for some $\rho<1$, $\alpha<1$ and $\beta>0$. Then, with probability at least $1-2\exp\left \{-c(m_2+r)\right  \}$, we have that
     \begin{equation*}
     \left \Vert V\left (\hat A -A_0\right )\right \Vert _2^{2}\lesssim 
          \sigma^{2}(m_2+r)\, \rank(V A_0).
     \end{equation*}
     \end{Theorem}
     \begin{proof}
     This is a consequence of Theorem \ref{thmr1}. 
     \end{proof}
     Let us now compare condition \eqref{r22} with the conditions obtained in 
      \cite{bunea-optimal,giraud-low}. In  \cite{bunea-optimal} the  authors introduce a new rank-penalised estimator and consider both cases when the variance of the noise is known or not. In the case of known variance of the noise, in \cite{bunea-optimal}, minimax optimal bounds on the mean squared errors are established (it does not need growth restriction on $\lambda$ and, thus, applies to all $\rank(V A_0 )$).
      In the case when the variance of the noise is unknown,  un unbiased estimator of $\sigma$ is proposed. This estimator requires an assumption on the dimensions of the problem. In particular it requires $m_2(n-r)$ to be large, which holds whenever $n\gg r$ or $n-r\geq 1$ and $m_2$ is large. This condition excludes an interesting case $n=r\ll m_2$. On the other hand \eqref{r22} is satisfied for $n=r\ll m_2$ if 
      \begin{equation*} 
         \rank(A_0)\lesssim n
          \end{equation*}
          where we used $\rank(VA_0)\leq r\wedge \rank(A_0)$.
          
          The method of \cite{giraud-low} requires the following condition to be satisfied
          \begin{equation} \label{giraud}
             \rank(A_0)\leq \dfrac{C_1 (n\,m_2-1)}{C_2\left (\sqrt{m_2}+\sqrt{r}\right )^{2}}
              \end{equation}
with some constants $C_1<1$ and $C_2>1$. This condition is quite similar to condition \eqref{r22}. Note  that, as $\rank(VA_0)\leq  \rank(A_0)$, condition \eqref{r22} is weaker than \eqref{giraud}. To the opposite of \cite{giraud-low}, our results are valid for all $A_0$ provided that
\begin{equation*} 
    r\leq \dfrac{\rho^{2}n\,m_2}{2\gamma^{2}\left (\sqrt{m_2}+\sqrt{r}\right )^{2}}.
     \end{equation*}
For large $m_2 \gg n$, this condition roughly mean that $n>cr$ for some constant$c$.

\section{Simulations}
\label{simulations}

In this section, we give empirical results that confirms our
theoretical findings. We illustrate the fact that using the Frobenius
norm instead of the square Frobenius norm as a goodness-of-fit
criterion makes the optimal smoothing parameter $\lambda$ independent
of the noise level, allowing for a better stability of the procedure
with respect to the noise level, as compared to other state-of-the-art
procedures. We focus on the matrix regression problem only, since
our conclusions are the same for matrix completion. We compare in particular the following procedures:
\begin{equation}
  \label{eq:procedure1}
  \text{argmin}_A \Big\{ \frac 12 \norm{U - V A}_2^2 + \lambda \norm{A}_1 \Big\},
\end{equation}
which is based on the classical least-squares penalized by the trace norm,
\begin{equation}
  \label{eq:procedure2}
  \text{argmin}_A \Big\{ \norm{U - V A}_2 + \lambda \norm{A}_1 \}
\end{equation}
which uses trace norm penalization with square-root least squares, and
\begin{equation}
  \label{eq:procedure3}
 \text{argmin}_A  \Big\{ \norm{U - V A}_2 + \lambda \norm{V A}_1 \Big\}
\end{equation}
which is the procedure introduced in this paper. We illustrate in particular the fact that~\eqref{eq:procedure2} and~\eqref{eq:procedure3}, which are based on a goodness-of-fit using the Frobenius norm instead of the squared Frobenius norm, provide a choice of $\lambda$ which is independent of the noise level $\sigma$.

\subsection{Optimization algorithms}
\label{sec:optimization}

In this section, we describe the convex optimization algorithms used for solving 
problems~\eqref{eq:procedure1}, \eqref{eq:procedure2} and~\eqref{eq:procedure3}.
For this we need to introduce the proximal operator \cite{bauschke2011convex} $\prox_g$ of a convex, proper, low-semicontinuous function $g$, given by
\begin{equation*}
  \prox_g(W) = \argmin_{Y} \Big\{ \frac {1}{2} \norm{W -
    Y}_2^2 + g(Y) \Big\}.
\end{equation*}
In the algorithms described below, we need to compute such proximal operator for specific functions. The proximal operator of the trace norm is given by spectral soft-thresholding, namely
\begin{equation*}
  \prox_{t g}(W) = \cS_{t}(W) \quad \text{ for } \quad g(W) = \| W \|_1
\end{equation*}
for any $t > 0$, where
\begin{equation*}
  \cS_t(W) = U_{W} \diag[( \sigma_1(W) - t)_+ \; \cdots \; (
  \sigma_{\rank(W)}(W) - t)_+] V_{W}^\top,
\end{equation*}
with $U_{W} \diag[ \sigma_1(W) \; \cdots \; \sigma_{\rank(W)}(W)]
V_{W}^\top$ the singular value decomposition of $W$, with the
columns of $U_{W}$ and $V_{W}$ being the left and right singular
vectors of $W$, and $\sigma_1(W) \geq \cdots \geq
\sigma_{\rank(W)}(W)$ its singular values. 

Problem~\eqref{eq:procedure1} is solved using accelerated proximal gradient, also known as Fista~\cite{beck2009fast}, since the loss is gradient-Lipschitz. Fista allows to minimize an objective of the form
\begin{equation*}
	F(A) = f(A) + g(A),
\end{equation*}
where $f$ is smooth (gradient-Lipshitz) with Lipschitz constant $L = \|V \|_\infty$ (the operator norm of $V$) and $g$ is prox-capable. In our setting we consider $f(A) = \frac 12 \norm{U - V A}_2^2$ and $g(A) = \lambda \norm{A}_1$, so that $\grad f(A) = V^\top (V A - U)$ and $\prox_{t g}(A) = \cS_t(A)$. The Fista algorithm is described in Algorithm~\ref{alg:fista} below. In our experiments we used backtracking linesearch, instead of fixing the step-size constant and equal to $1 / L$.
\begin{algorithm}
  \caption{Fista}
  \begin{algorithmic}[1]
    \REQUIRE Starting points $B^1 = A^0$, Lipschitz constant $L > 0$ for $\grad f$, $t_1 = 1$%
    \FOR {$k=0,1,2,\ldots$}%
    \STATE $A^{k} \gets \prox_{L^{-1} g} (B^k - \frac{1}{L} \grad
        f(B^k))$
    \STATE $t_{k+1} = \frac{1 + \sqrt{1 + 4 t_k^2}}{2}$
    \STATE $B^{k+1} = A^{k} + \frac{t_k - 1}{t_{k+1}}
        (A^{k} - A^{k-1})$
    \ENDFOR
    \RETURN $A^k$
  \end{algorithmic}
  \label{alg:fista}
\end{algorithm}

\noindent
Problem~\eqref{eq:procedure2} is solved using a primal-dual algorithm~\cite{chambolle2011first}, see Algorithm~\ref{alg:cp}. It allows to minimize an objective of the form
\begin{equation}
	F(A) = f(K A) + g(A),
\end{equation}
where both $f$ and $g$ are prox-capable (with $f$ non-smooth) and $K$ a linear operator. In our setting we choose this time $K = V$, $f(A) = \norm{A - U}_2$ and $g(A) = \lambda \norm{A}_1$. It is easily proved that
\begin{equation*}
	\prox_{t f}(A) = 
	\begin{cases}
		U &\text{ if } \norm{A - U}_2 \leq t \\
		A - t \frac{A - U}{\norm{A - U}_2} &\text{ if } \norm{A - U}_2 > t,
	\end{cases}
\end{equation*}
which allows to instantiate Algorithm~\ref{alg:cp} for problem~\eqref{eq:procedure2}, using also the Moreau's identity $\prox_{f^*}(A) - A - \prox_f(A)$, see~\cite{bauschke2011convex}, where $f^*$ is the Fenchel conjugate of $f$. In Algorithm~\ref{alg:cp} we use the heuristics described in~\cite{chambolle2011first} to choose the step-sizes $\eta$ and $\tau$.

\begin{algorithm}
  \caption{Primal-dual algorithm}
  \begin{algorithmic}[1]
    \REQUIRE Starting points $A^0, \bar A^0, Z^0$, step-sizes $\eta, \tau > 0$ such that 
    \FOR {$k=0,1,2,\ldots$}%
    \STATE $Z^{k+1} \gets \prox_{\eta f^*}(Z^k + \eta V \bar A^k)$ \\
    \STATE $A^{k+1} \gets \prox_{\tau g}(A^k - \tau V^\top Z^{k+1})$ \\
    \STATE $\bar A^{k+1} \gets A^{k + 1} + \theta (A^{k+1} - A^k)$
    \ENDFOR
    \RETURN $A^k$
  \end{algorithmic}
  \label{alg:cp}
\end{algorithm}

\noindent
Problem~\eqref{eq:procedure3} is solved using parallel splitting~\cite{bauschke2011convex}. First, we need to reformulate the problem. Let us observe that if $\hat A$ is a solution to~\eqref{eq:procedure3}, then any $\hat A + B$
with $B \in \ker(V)$, where $\ker(V) = \{ A \in \R^{m_1 \times m_2} :
V A = 0 \}$, is also a solution. Thus, we will solve the problem on a splitted variable $W = VA$.
We define the linear space $\col(V) = \{ W
\in \R^{n \times m_2} : \exists A \in \R^{m_1 \times m_2}, V A = W
\}$. Then, we have
\begin{equation*}
  V \hat A = \argmin_{W \in \col(V)} \norm{U - W}_2 + \lambda
  \norm{W}_1,
\end{equation*}
so that we end up with the problem
\begin{equation}
  \label{eq:reformulatedsqrtmatrixregression}
  \mini \quad \norm{U - W}_2 + \lambda \norm{W}_1 +
  \delta_{\col(V)}(W),
\end{equation}
where $\delta_C(X)$ stands for the indicator function of a convex set
$C$, defined by $\delta_C(X) = 0$ when $X \in C$ and $\delta_C(X) =
+\infty$ when $X \notin C$. Then, we solve~\eqref{eq:reformulatedsqrtmatrixregression} using parallel
splitting~\cite{bauschke2011convex}. Each function
in~\eqref{eq:reformulatedsqrtmatrixregression} are
prox-capable. Let us define
\begin{equation*}
  f_1(W) = \norm{U - W}_2, \quad f_2(W) = \lambda \norm{W}_1, \quad
  f_3(W) = \delta_{\col(V)}(W).
\end{equation*}
We gave above $\prox_{f_1}$ and $\prox_{f_2}$. We have that
\begin{equation*}
  \prox_{f_3}(W) = \cP_{\col(V)}(W) = V (V^\top
V)^{\dagger} V^\top W,
\end{equation*}
where $\cP_{\col(V)}$ is the projection operator onto the set
$\col(V)$, and where $Z^\dagger$ stands for the pseudo-inverse of $Z$. The parallel splitting algorithm is described
in Algorithm~~\ref{alg:ps} below.

\begin{algorithm}
  \caption{Parallel splitting}
  \begin{algorithmic}[1]
    \REQUIRE Step-sizes $\gamma > 0$, $\tau_k \in [0, 2]$, initial
    values $W_{1}^0, W_{2}^0, W_{3}^0$%
    \FOR {$k=0,1,2,\ldots$}%
    \STATE $P^k \leftarrow \frac 13 (W_{1}^k + W_{2}^k + W_{3}^k)$%
    \STATE $Z_{i}^k \leftarrow \prox_{\gamma f_i}(W_{i}^k)$ for $i =
    1, 2, 3$%
    \STATE $Q^k \leftarrow \frac 13 (Z_{1}^k + Z_{2}^k + Z_{3}^k)$%
    \STATE $W_{i}^{k+1} \leftarrow W_{i}^k + \tau_k (2 Q^k - P^k -
    Z_{i}^k)$ for $i = 1, 2, 3$
    \ENDFOR
    \RETURN $P^k$
  \end{algorithmic}
  \label{alg:ps}
\end{algorithm}

\noindent
Convergence is guaranteed for $\tau_k \in [0, 2]$ such that $\sum_{k \geq 0}
\tau_k (2 - \tau_k) = +\infty$, see~\cite{bauschke2011convex}, we simply choose $\tau_k = 1.9$ in our experiments. An
alternative (but somewhat less direct) method for
solving~\eqref{eq:reformulatedsqrtmatrixregression} is to write an
equivalent conic formulation, and smooth the primal objective by
adding a strongly convex term. Then, the corresponding dual problem can be
solved using first order techniques. This method, called
\texttt{TFOCS}, is the one described in~\cite{becker2010templates} for
solving general convex cone problems.

\subsection{Numerical illustration}

We give several numerical illustrations. First, we show that the optimal choice of $\lambda$ is almost independent of the noise level for the procedures~\eqref{eq:procedure2}, \eqref{eq:procedure3}, while it needs to be increased with $\sigma$ for procedure~\eqref{eq:procedure1}. This fact is illustrated in Figures~\ref{fig:expe1} and~\ref{fig:expe2}. Then, we compare the best prediction errors (among prediction errors obtained for several $\lambda$) of solutions of problems~\eqref{eq:procedure1},~\eqref{eq:procedure2} and~\eqref{eq:procedure3}. This is illustrated in Tables~\ref{tab:pred_errors} and~2.

We simulate data as follows. We pick at random $A_1$ and $A_2$ as, respectively, $m_1 \times
r$ and $m_2 \times r$ matrices with $N(0, 1)$ i.i.d
entries, and we fix $A_0 = A_1 A_2^\top$, which is a $m_1
\times m_2$ matrix with rank $r$ a.s. We pick at random
a $n \times m_1$ matrix $V$, with lines $V_i \in \R^{m_1}$, $i = 1,
\ldots, n$, distributed as a centered Gaussian vectors with covariance
equal to the Toeplitz matrix $\Sigma = (\rho^{-|i-j|})_{1 \leq i, j
  \leq m_1}$. We finally compute $U = V A_0 + \sigma E$, where the
noise matrix $E$ contains $N(0, 1)$ i.i.d entries and $\sigma > 0$ is
the standard deviation. 

We consider the setting $n = 1000$, $m_1 = 200$, $m_2 = 100$, $r = 10$ and $\rho=0.5$, called ``experiment~1'' in Figures and Tables, while we choose $n = 200$, $m_1 = 100$, $m_2 = 400$ and other parameters unchanged for ``experiment~2''.

In Figure~\ref{fig:expe1},~2  Table~\ref{tab:pred_errors},~2 we consider values of $\sigma$ in $\{ 0.1, 0.5, 1.0, 5.0 \}$, and for each value of $\sigma$ we plot the prediction error $\norm{V(\hat A_\lambda - A_0)}_2$ for a parameter $\lambda$ in a grid. We repeat this 10 times, and plot each time the prediction error in Figure~\ref{fig:expe1} and print the average best prediction errors (and standard deviation) in Table~\ref{tab:pred_errors}.

The conclusion of this experiment is the following: we observe that the minimum of the prediction error is achieved for a parameter $\lambda$ that increases with $\sigma$ for procedure~\eqref{eq:procedure1}, while it is almost constant for procedures~\eqref{eq:procedure2} and~\eqref{eq:procedure3}. This confirms numerically the fact, when using square-root least-squares instead of least-squares, the optimal choice of~$\lambda$ can be done independently of the noise level.
We observe also that the minimum prediction errors of each procedure are of the same order for experiment~1, with a slight advantage for procedure~\eqref{eq:procedure3} for each considered values of~$\sigma$, while there is a strong advantage for procedure~\eqref{eq:procedure3} for experiment~2, which corresponds to the case where the number of tasks $m_2$ is larger than the sample size $n$. 

\begin{figure}[htbp]
	\centering
	\includegraphics[width=0.34\textwidth]{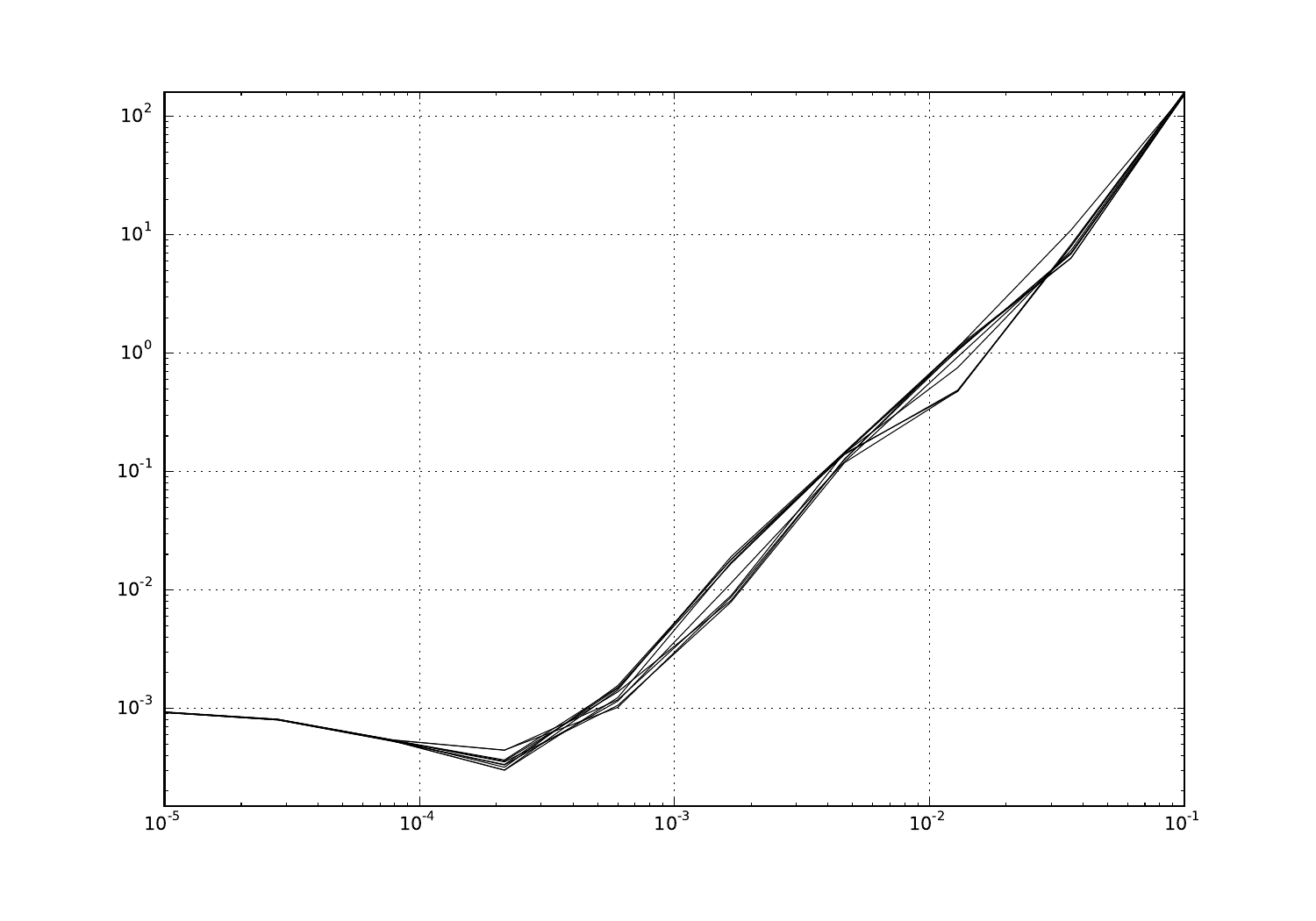}%
	\includegraphics[width=0.34\textwidth]{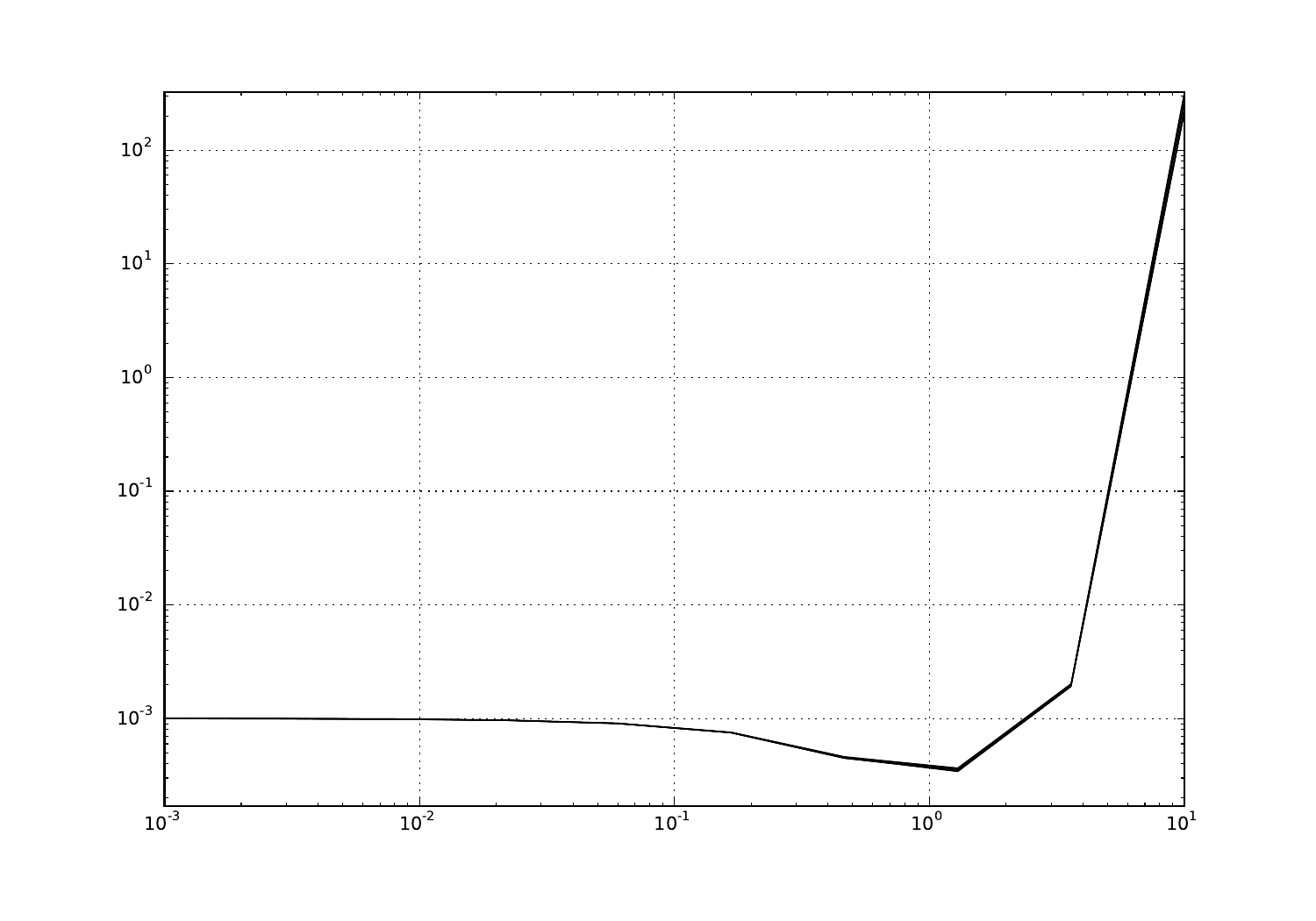}%
	\includegraphics[width=0.34\textwidth]{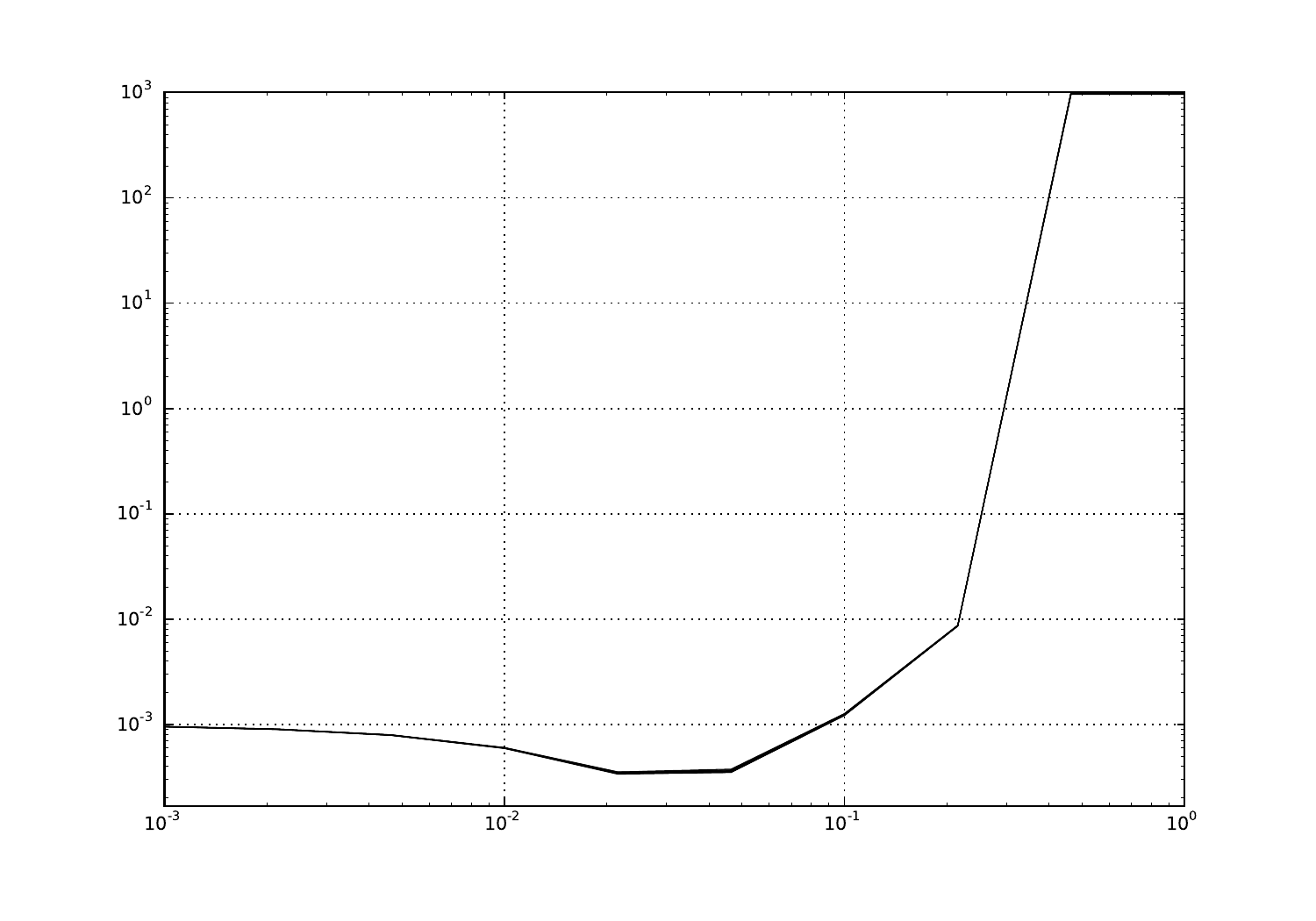}

	\includegraphics[width=0.34\textwidth]{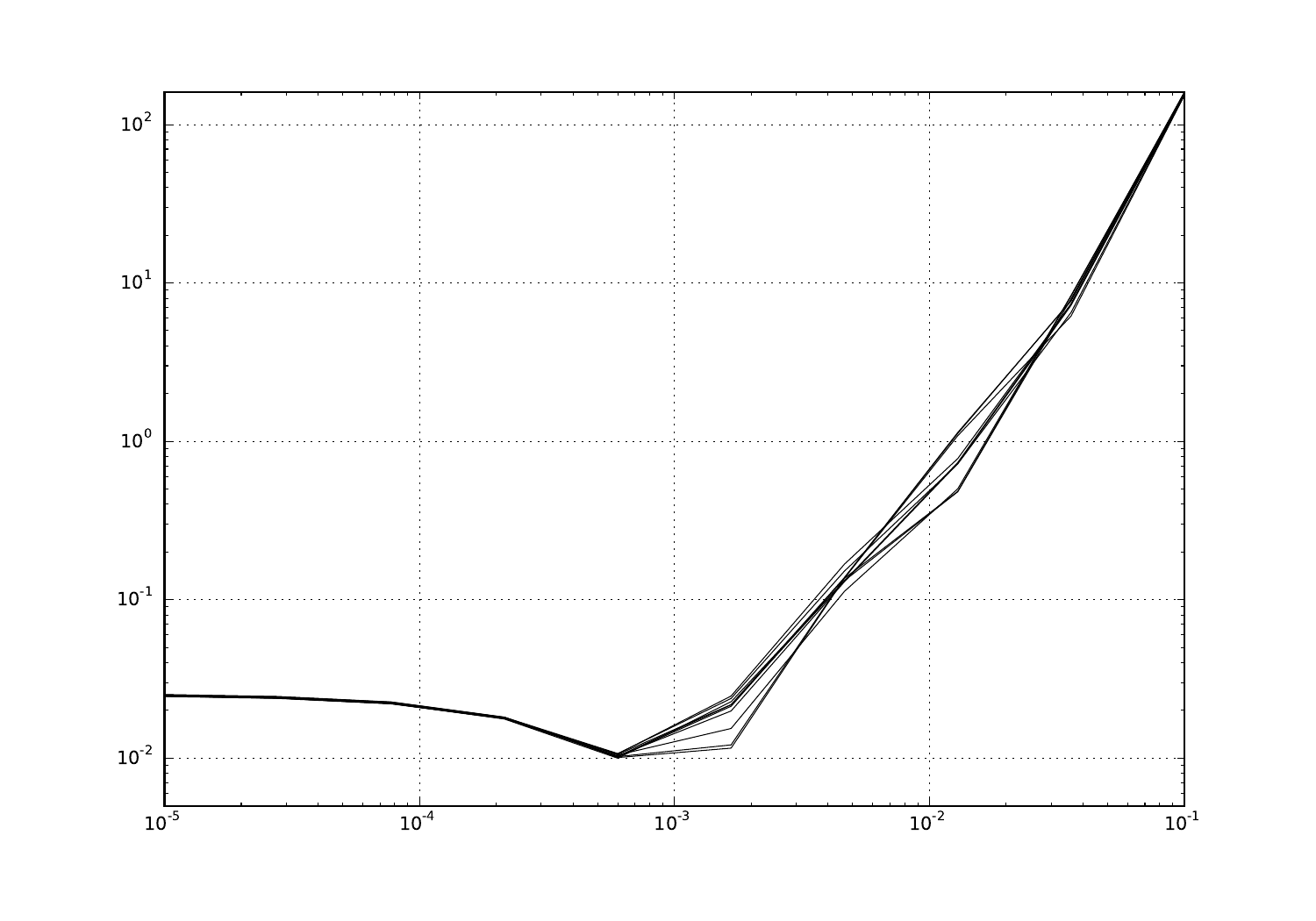}%
	\includegraphics[width=0.34\textwidth]{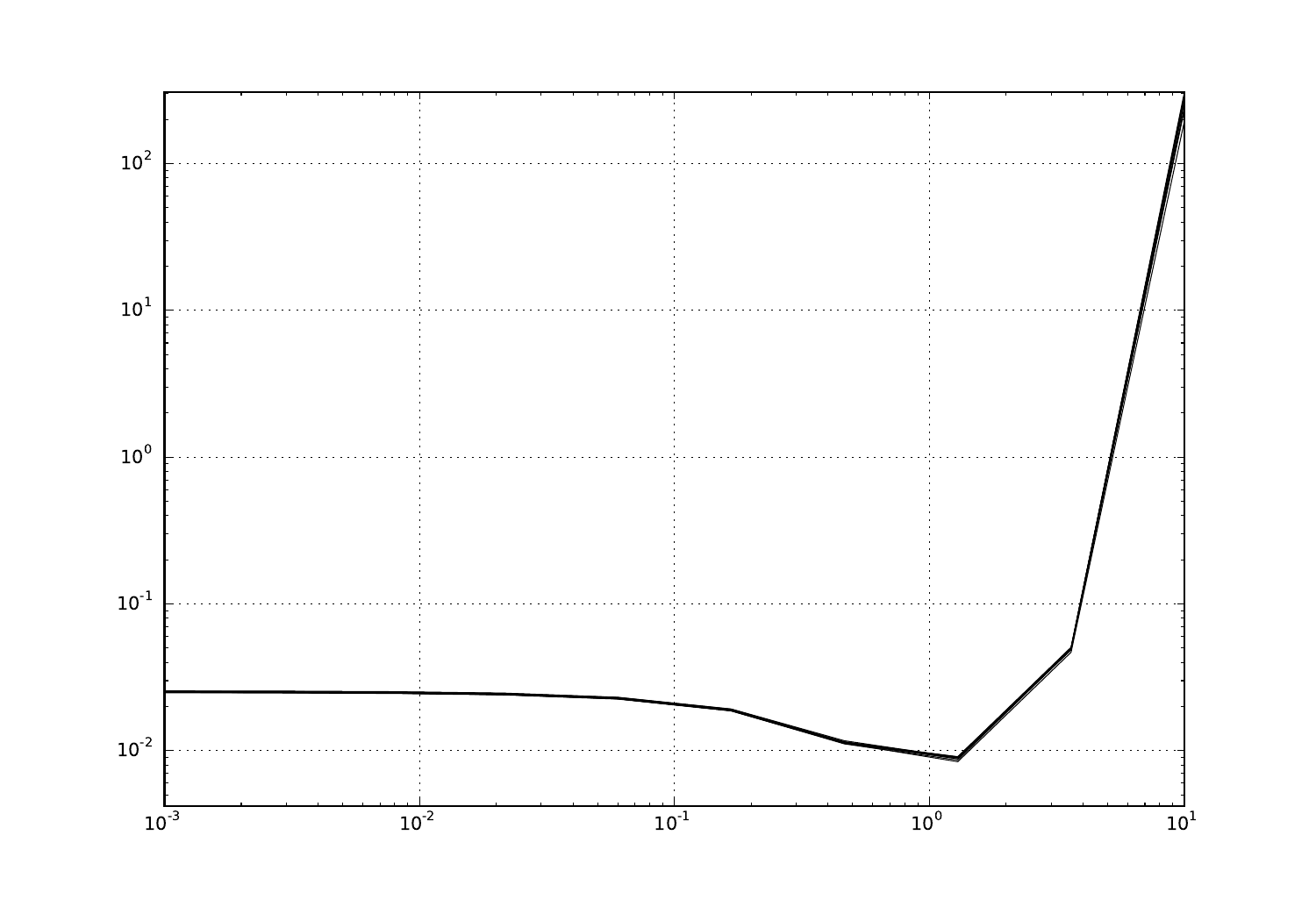}%
	\includegraphics[width=0.34\textwidth]{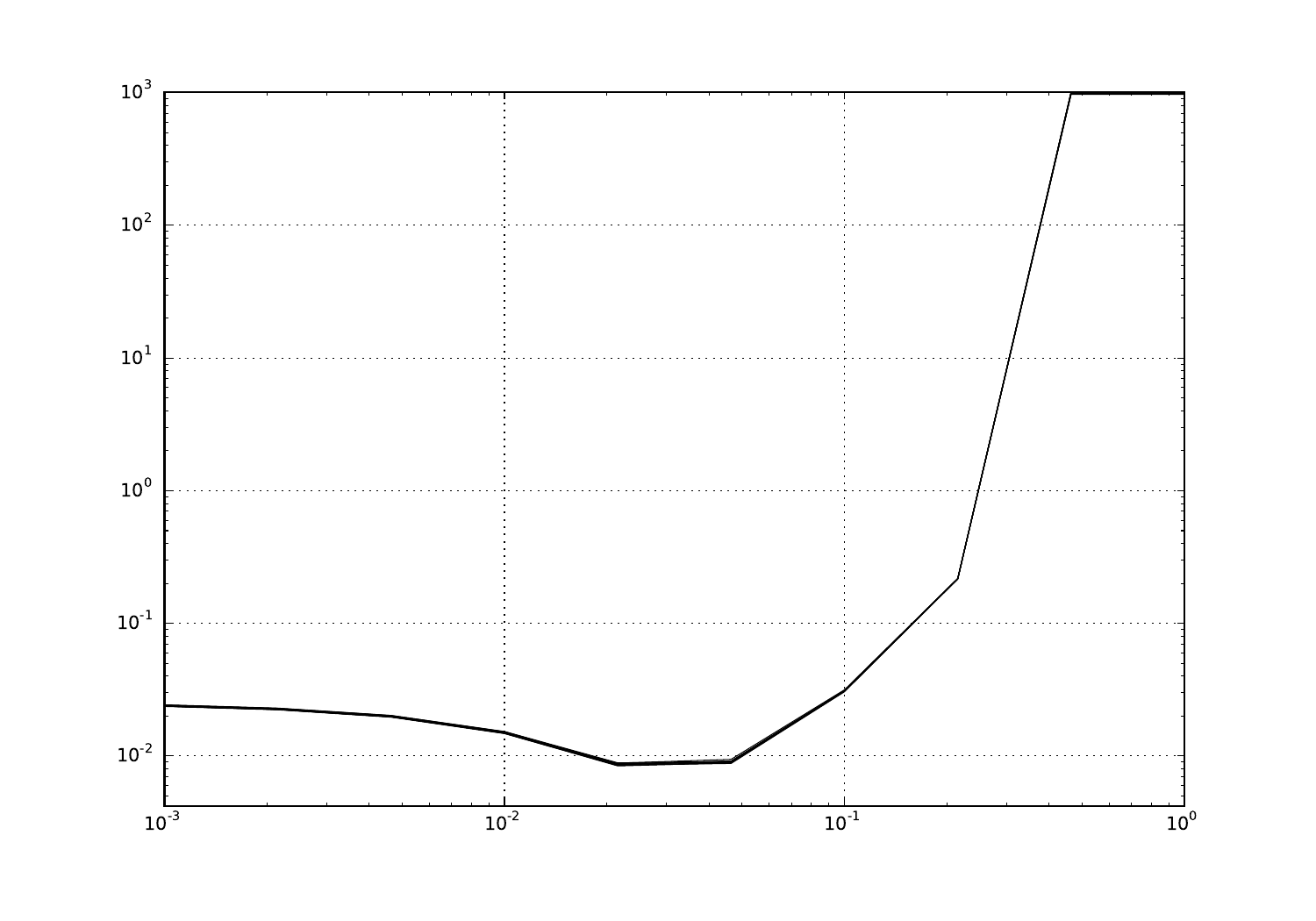}

	\includegraphics[width=0.34\textwidth]{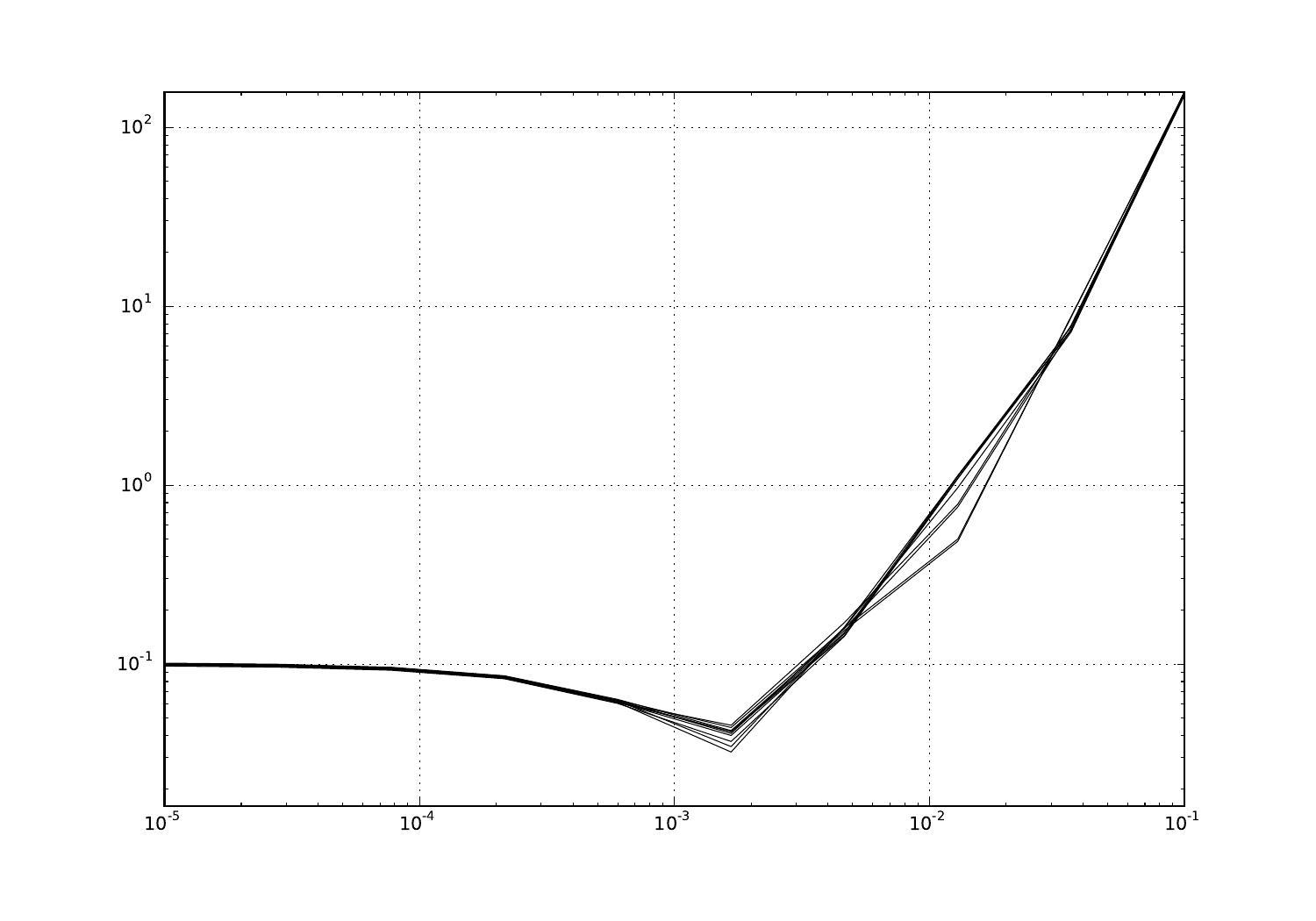}%
	\includegraphics[width=0.34\textwidth]{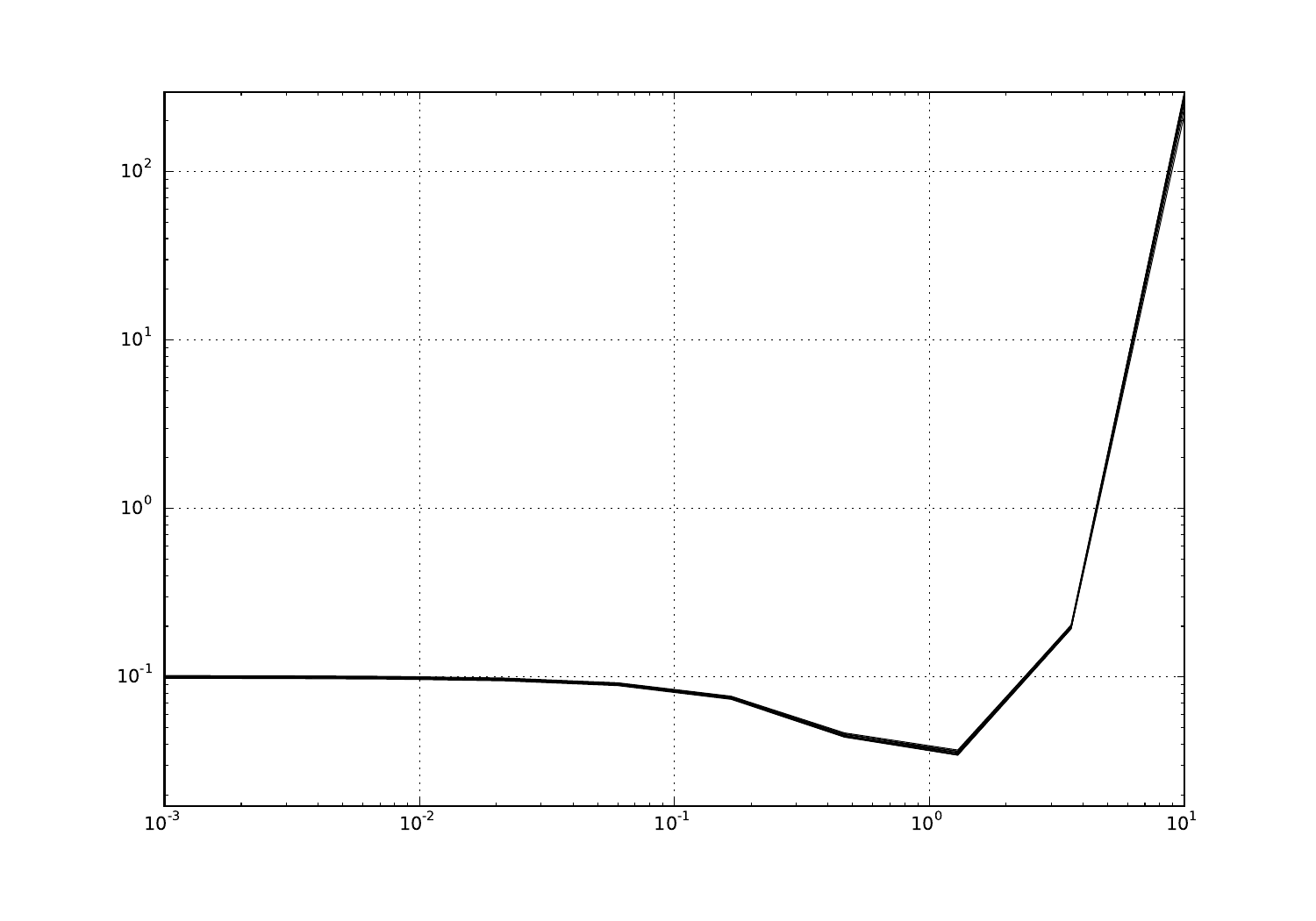}%
	\includegraphics[width=0.34\textwidth]{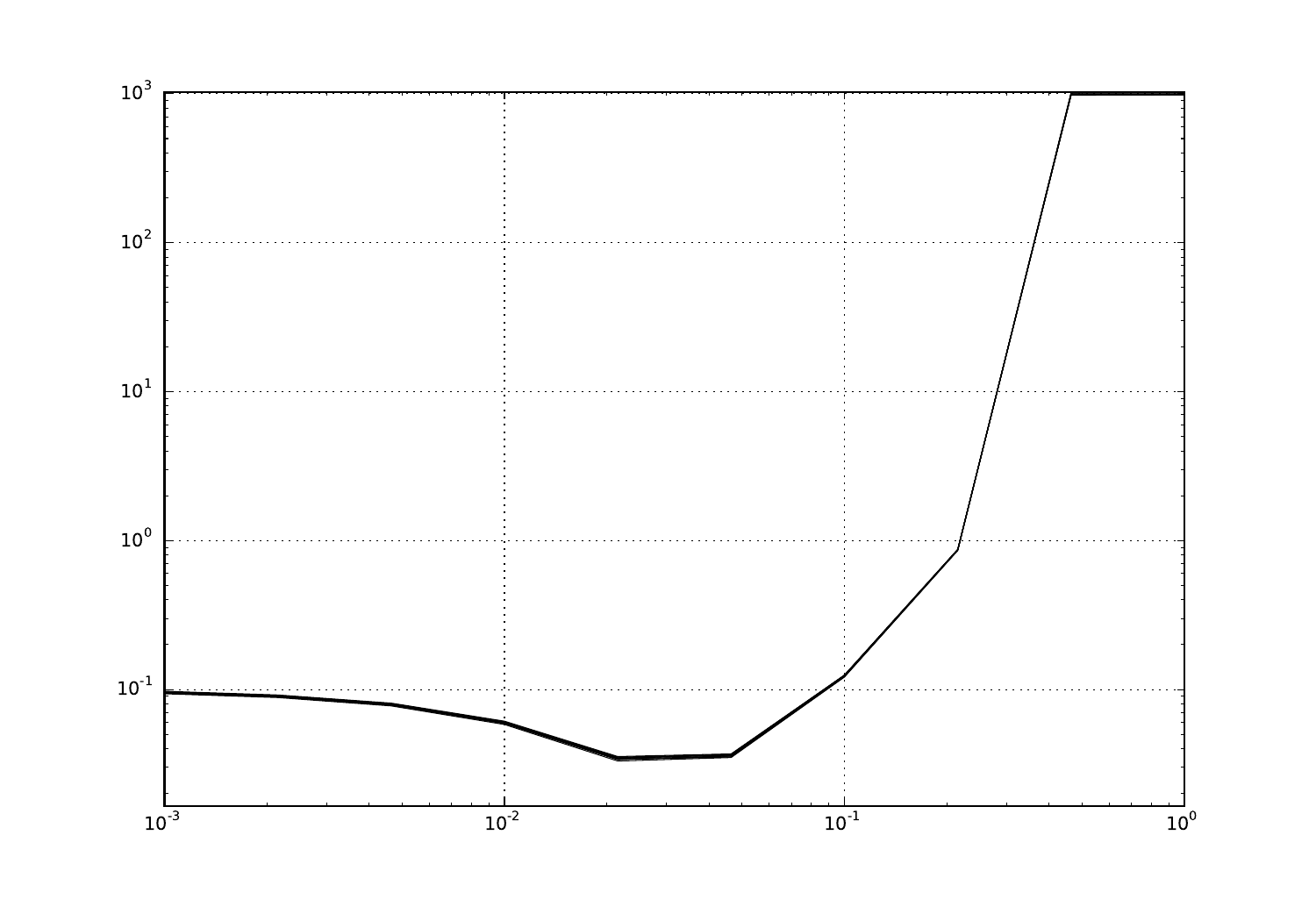}

	\includegraphics[width=0.34\textwidth]{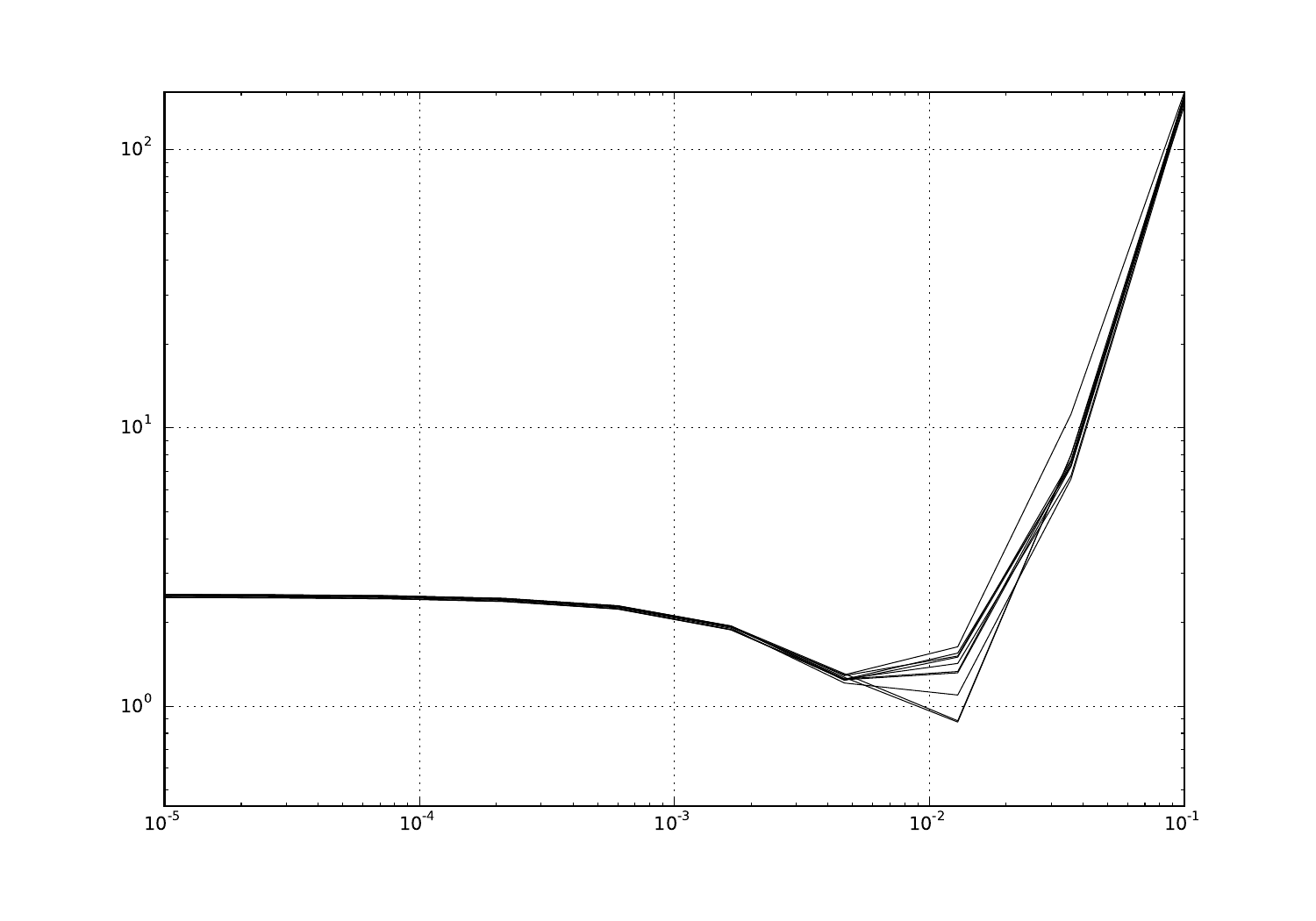}%
	\includegraphics[width=0.34\textwidth]{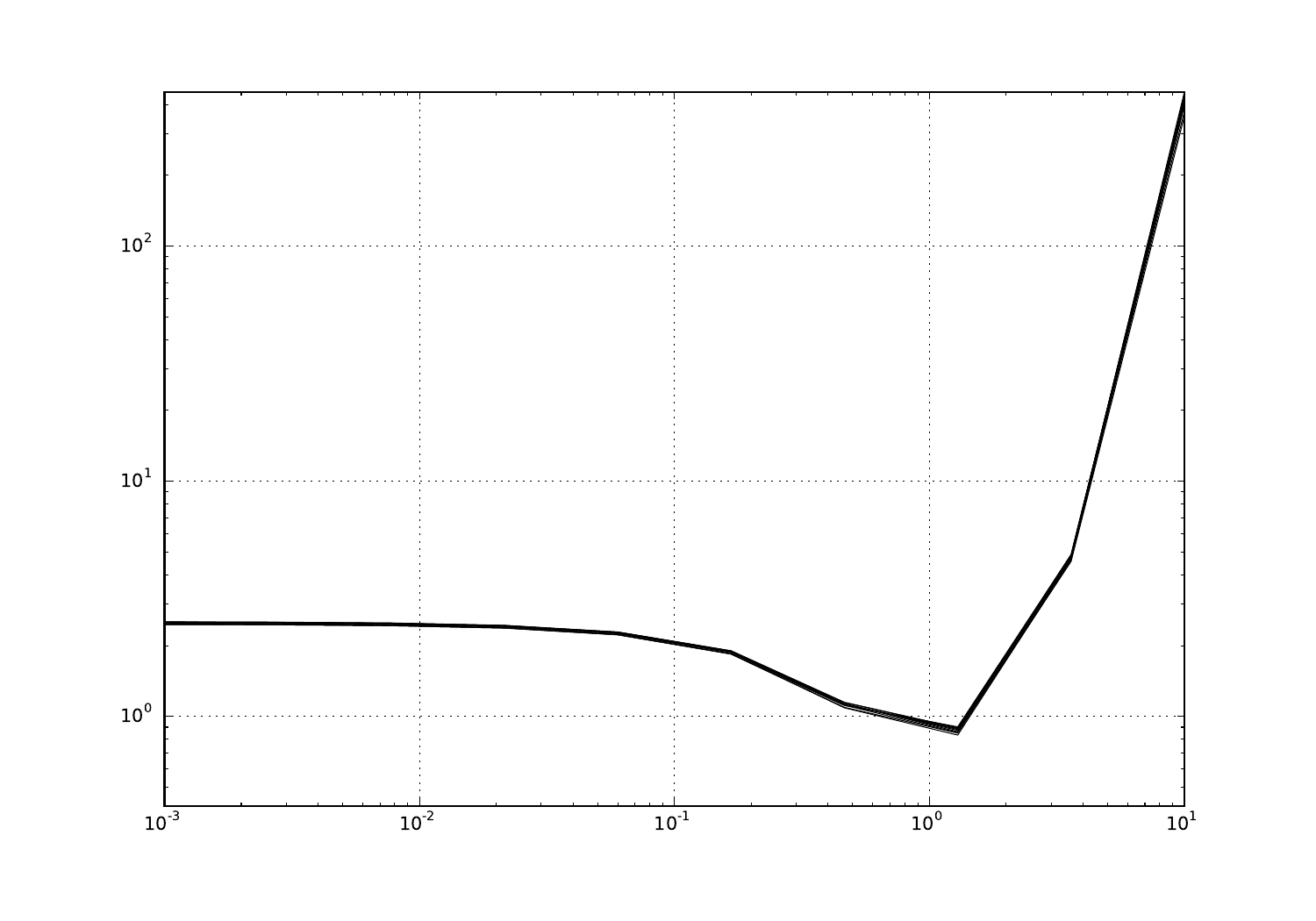}%
	\includegraphics[width=0.34\textwidth]{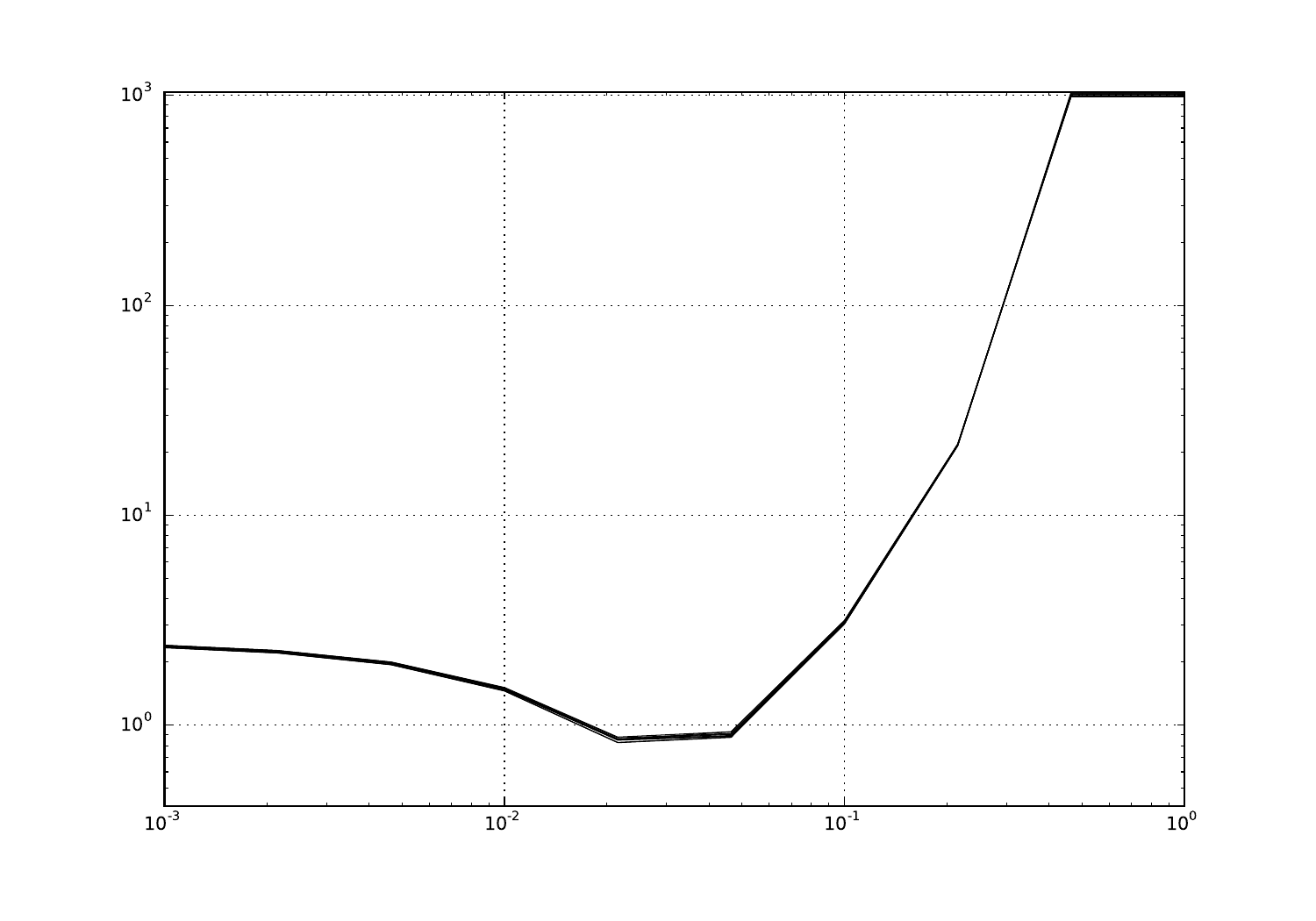}

	\caption{Prediction errors (y-axis) for experiment~1 (see text) for a varying $\lambda$ (x-axis) for procedure~\eqref{eq:procedure1} (first column), procedure~\eqref{eq:procedure2} (second column) and procedure~\eqref{eq:procedure3} (third column). 
	We plot the estimation errors over 10 simulated datasets (corresponding to a line in each figure), for an increasing noise level $\sigma=0.1$ (first line), $\sigma=0.5$ (second line), $\sigma=1.0$ (third line), $\sigma=5.0$ (fourth line).
	We can observe that the optimum $\lambda$ for~\eqref{eq:procedure1} increases with $\sigma$ (see the position of minimum along the first column), while it can be kept almost constant for procedures~\eqref{eq:procedure2} and~\eqref{eq:procedure3}}
	\label{fig:expe1}
\end{figure}

\begin{table}[htbp]
  \centering
  \footnotesize
  \noindent
  \begin{tabular}{lcccc}
    Noise level $\sigma$ & 0.1 & 0.5 & 1.0 & 5.0 \\ \hline
	Procedure~\eqref{eq:procedure1} & 3.56e-04  & 1.03e-02  & 4.01e-02  & 1.17e+00  \\
	& (4.90e-05) & (2.23e-04) & (4.02e-03) & (1.52e-01) \\
	Procedure~\eqref{eq:procedure2} & 3.54e-04  & 8.87e-03  & 3.54e-02  & 8.72e-01  \\
	& (8.66e-06) & (2.01e-04) & (8.34e-04) & (2.17e-02) \\
	Procedure~\eqref{eq:procedure3} & \textbf{3.47e-04} & \textbf{8.65e-03} & \textbf{3.43e-02} & \textbf{8.54e-01} \\
	& (5.16e-06) & (1.44e-04) & (6.73e-04) & (1.56e-02) \\
  \end{tabular}
  \caption{Average best prediction error (and standard deviation) for experiment~1 of the considered procedures for several values $\sigma$. Procedure~\eqref{eq:procedure3} introduced in this paper always leads to a slight improvement.}
  \label{tab:pred_errors}
\end{table}

\begin{figure}[htbp]
	\centering
	\includegraphics[width=0.34\textwidth]{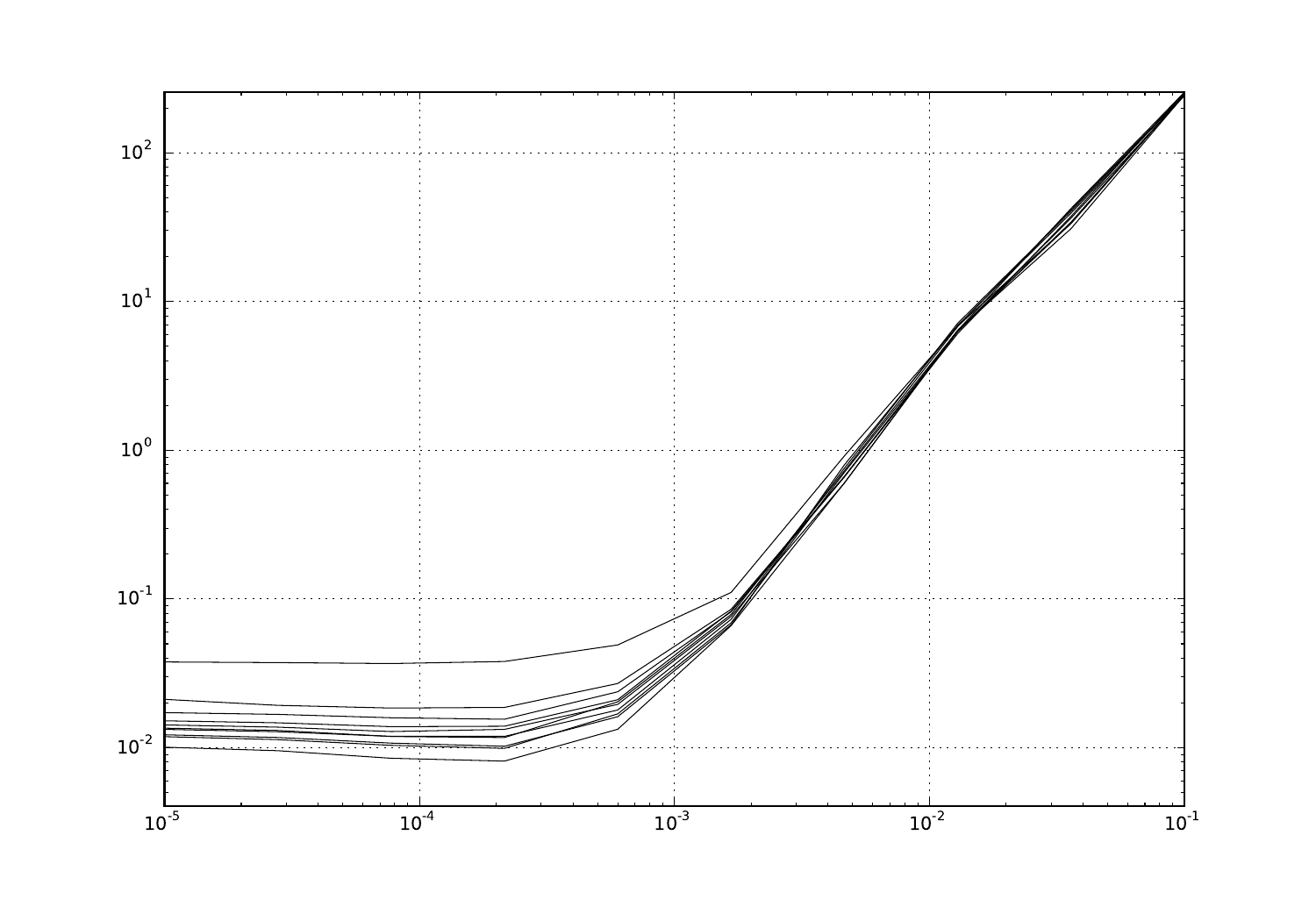}%
	\includegraphics[width=0.34\textwidth]{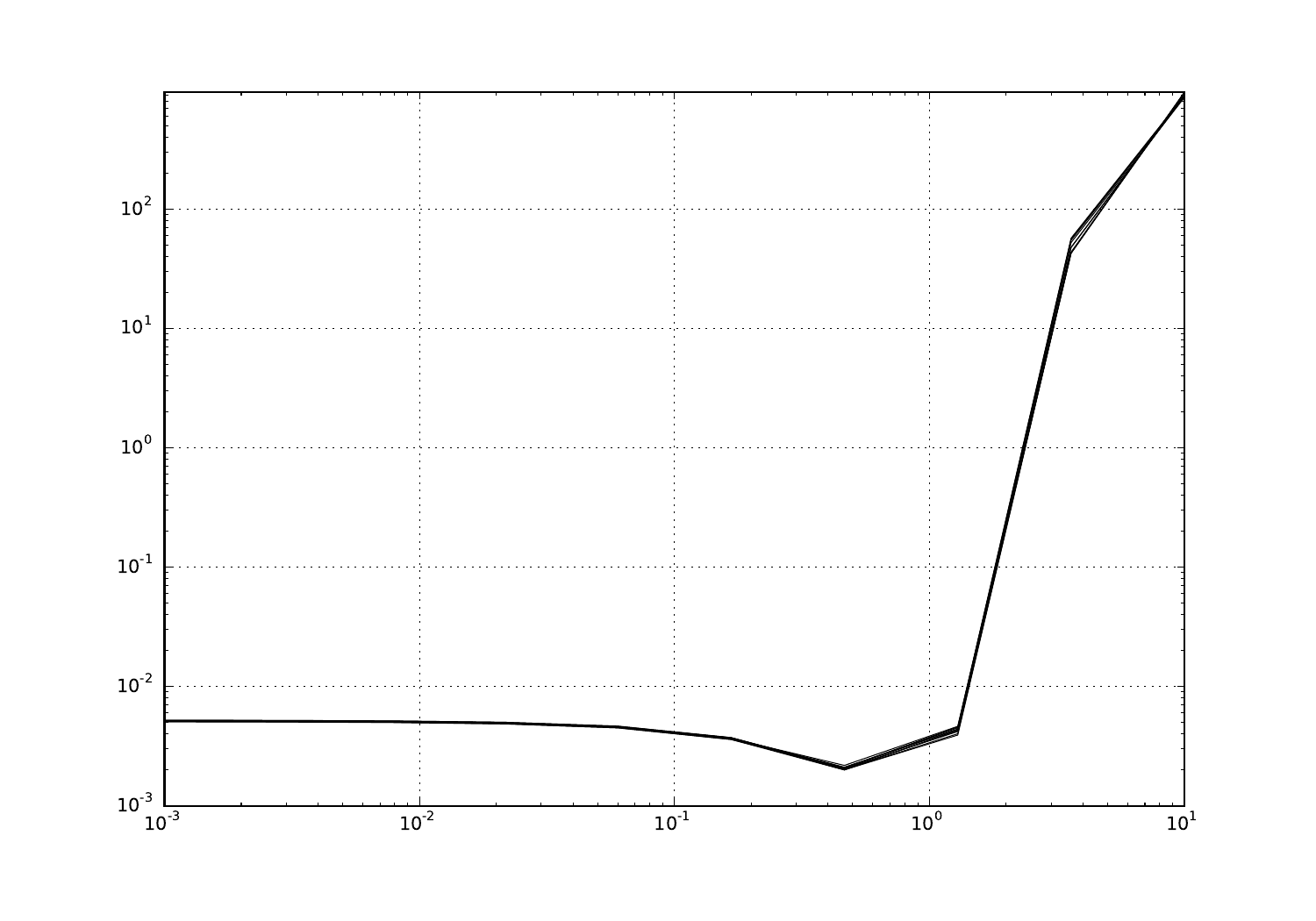}%
	\includegraphics[width=0.34\textwidth]{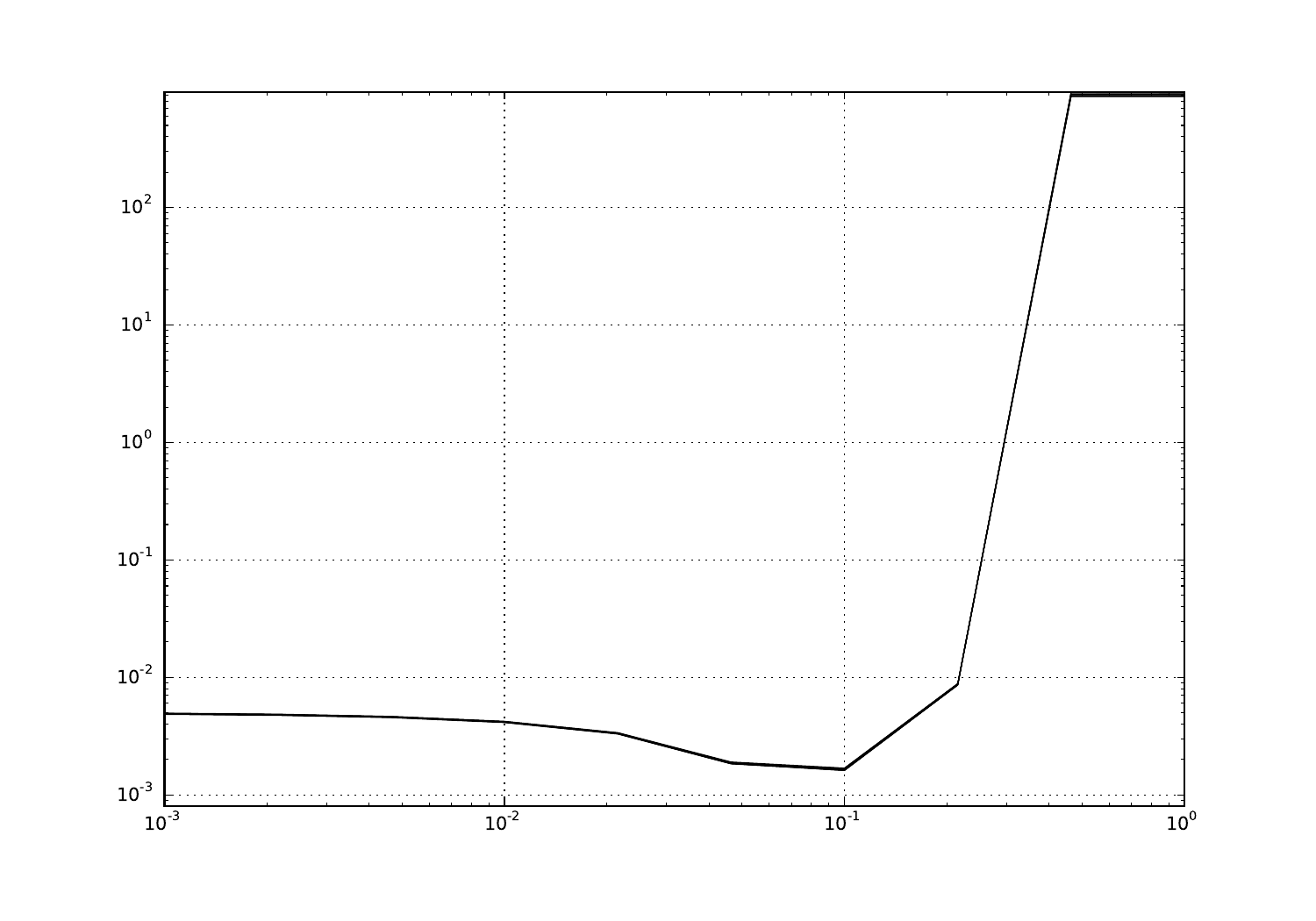}

	\includegraphics[width=0.34\textwidth]{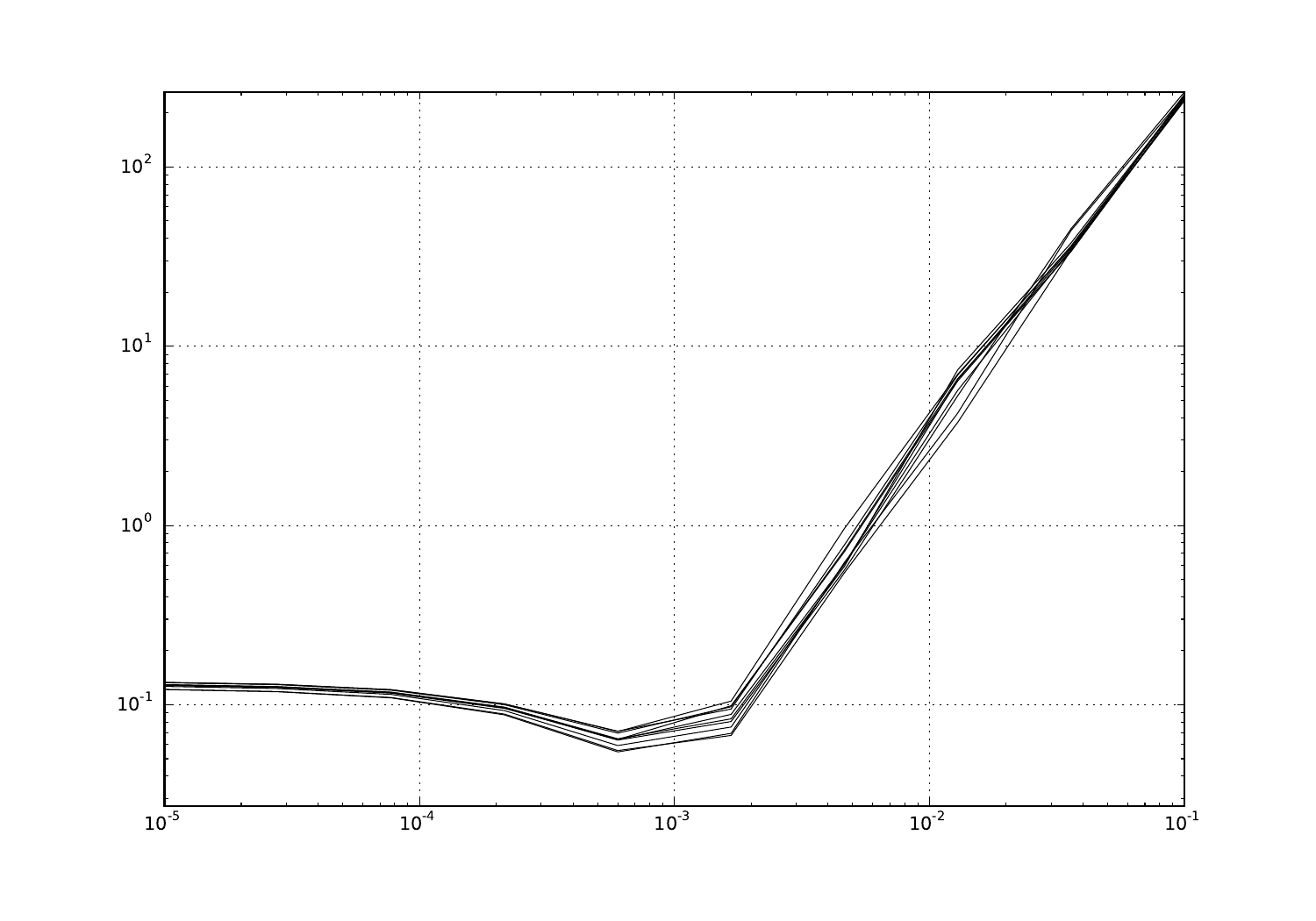}%
	\includegraphics[width=0.34\textwidth]{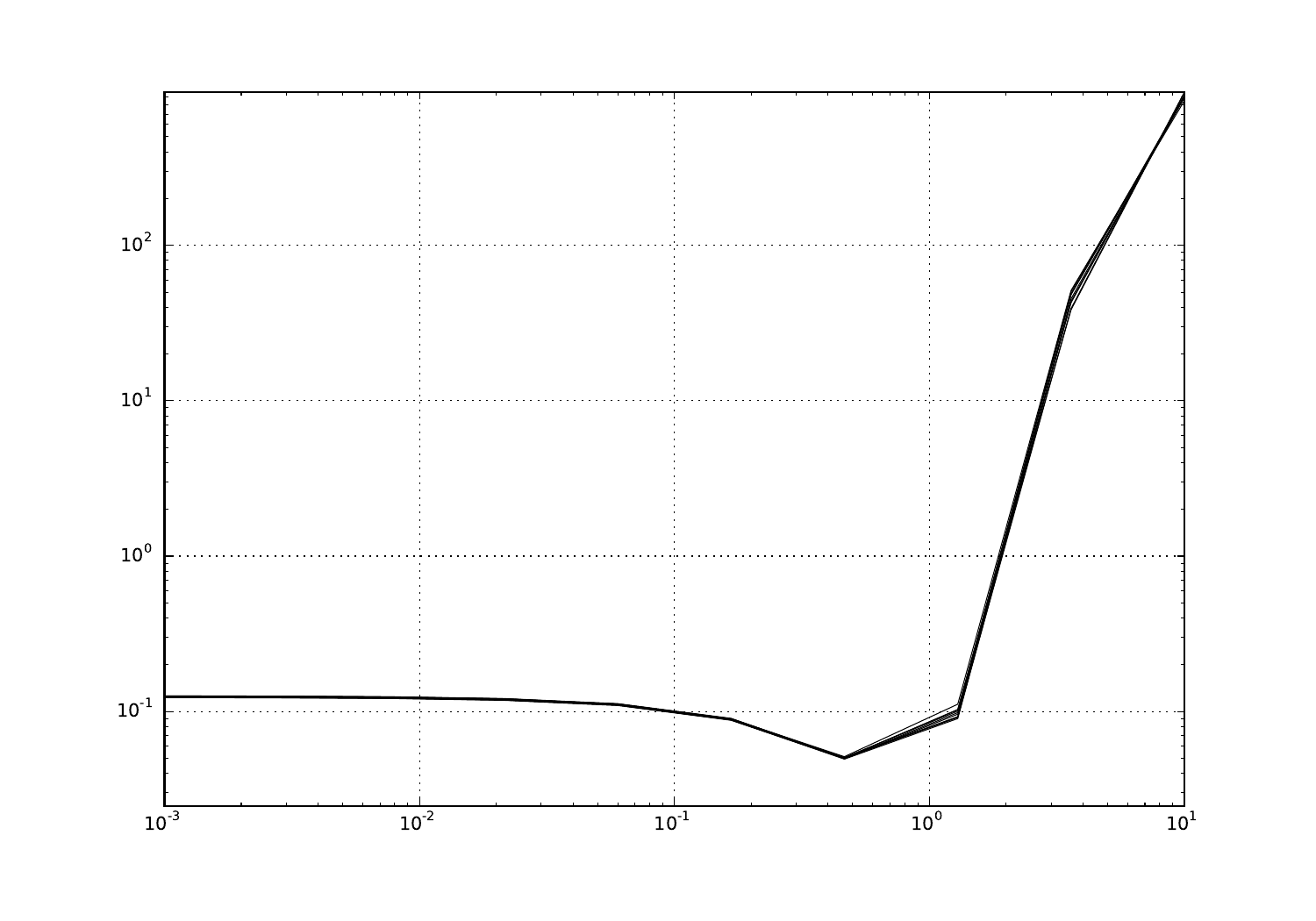}%
	\includegraphics[width=0.34\textwidth]{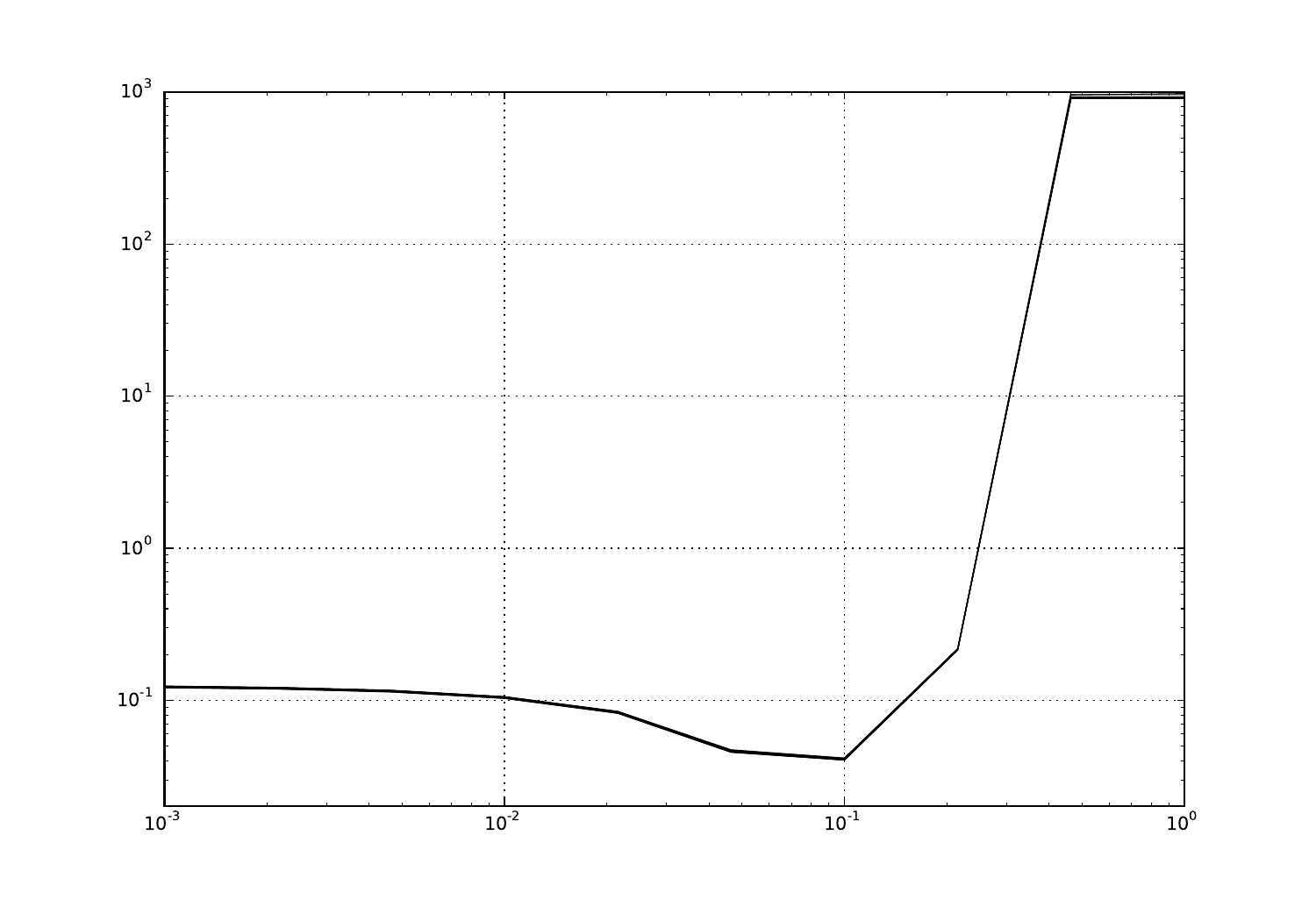}

	\includegraphics[width=0.34\textwidth]{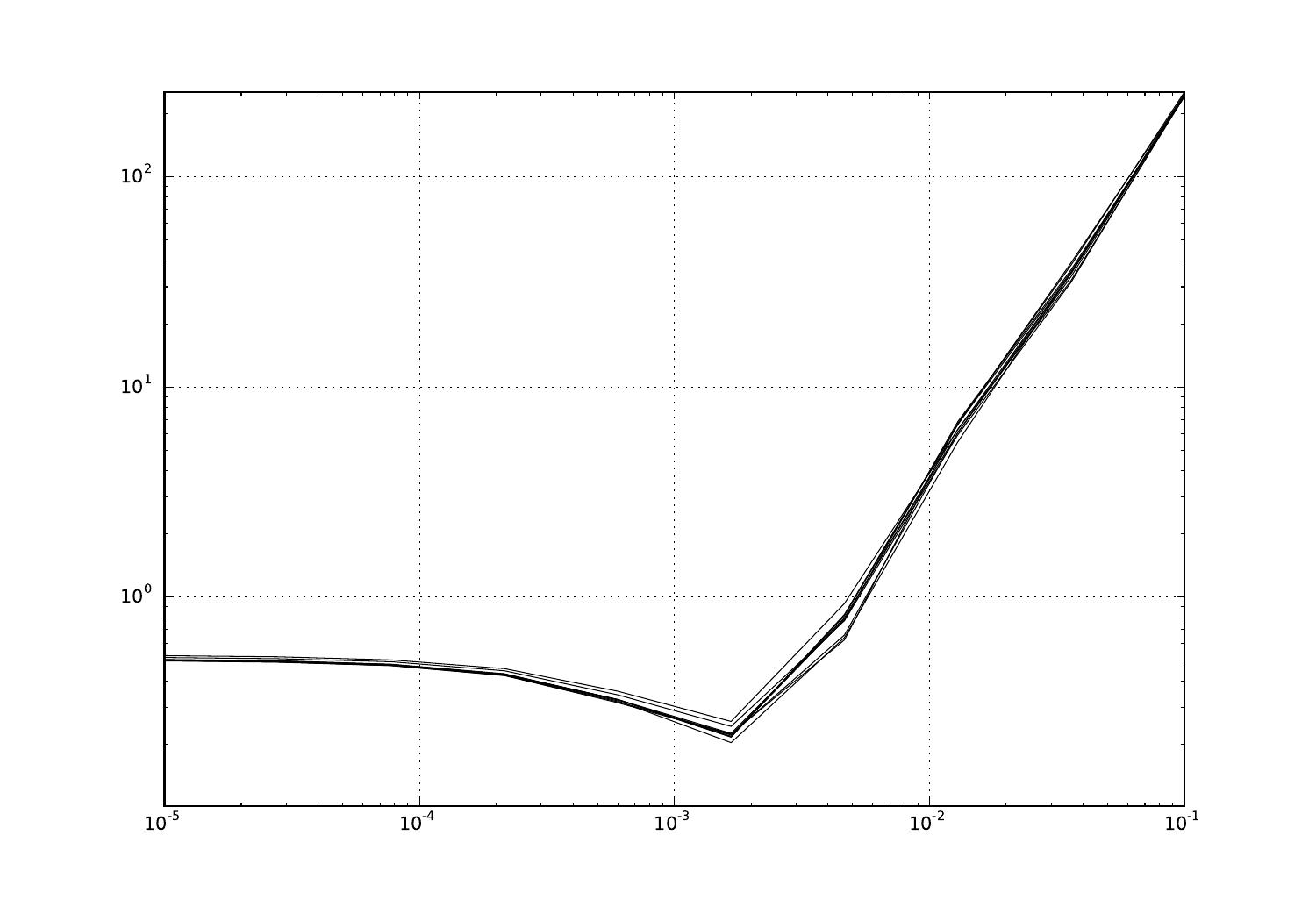}%
	\includegraphics[width=0.34\textwidth]{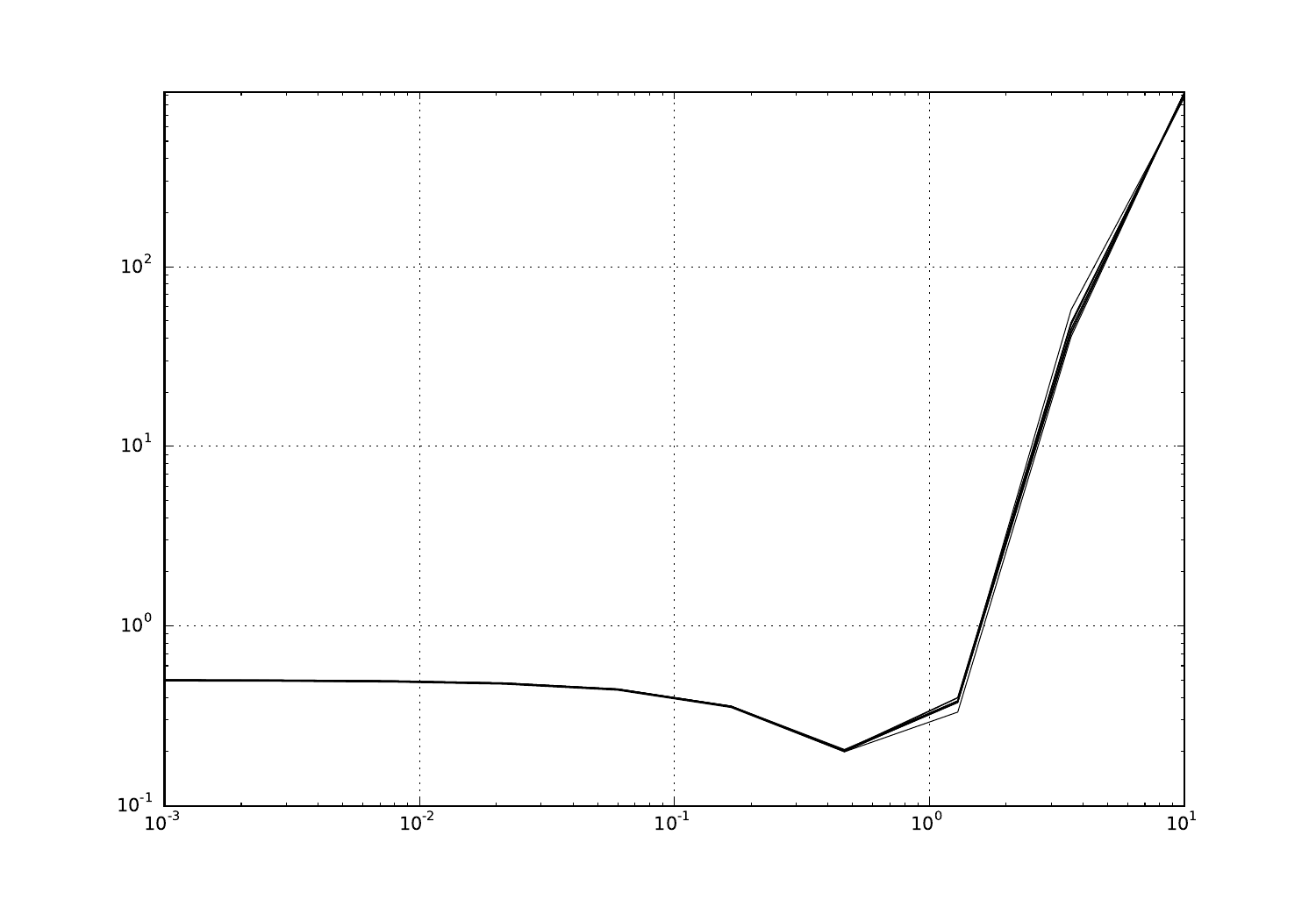}%
	\includegraphics[width=0.34\textwidth]{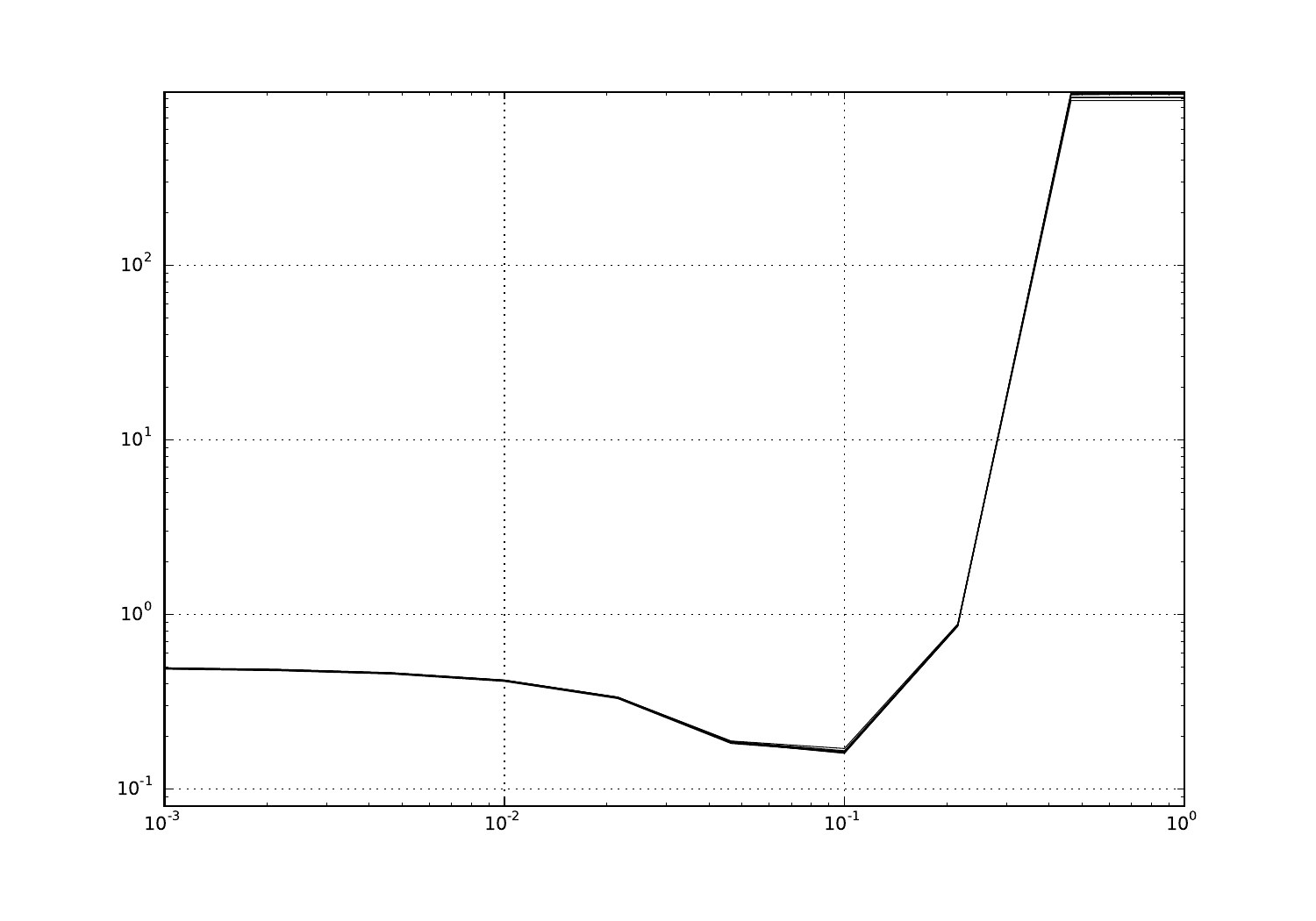}

	\includegraphics[width=0.34\textwidth]{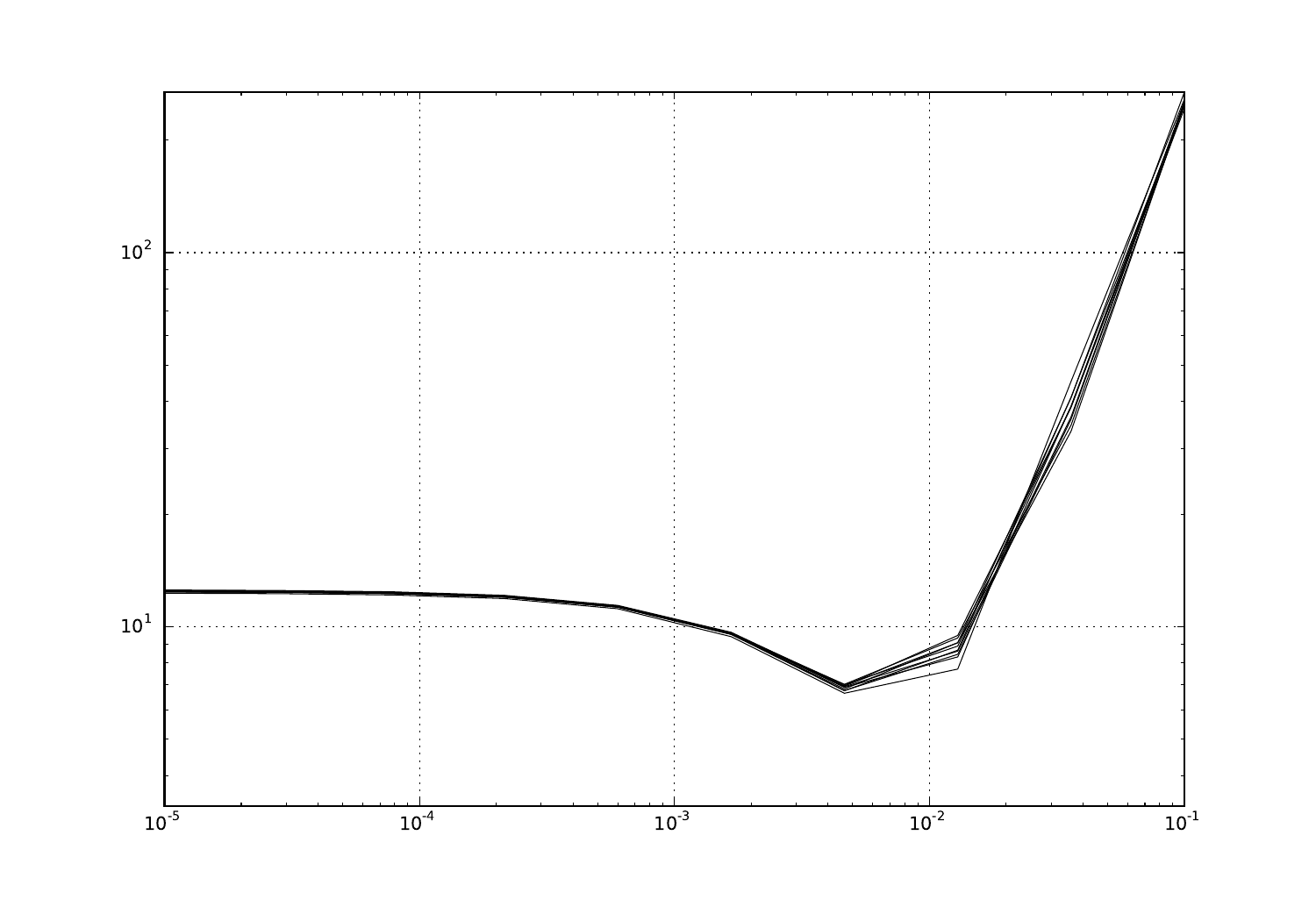}%
	\includegraphics[width=0.34\textwidth]{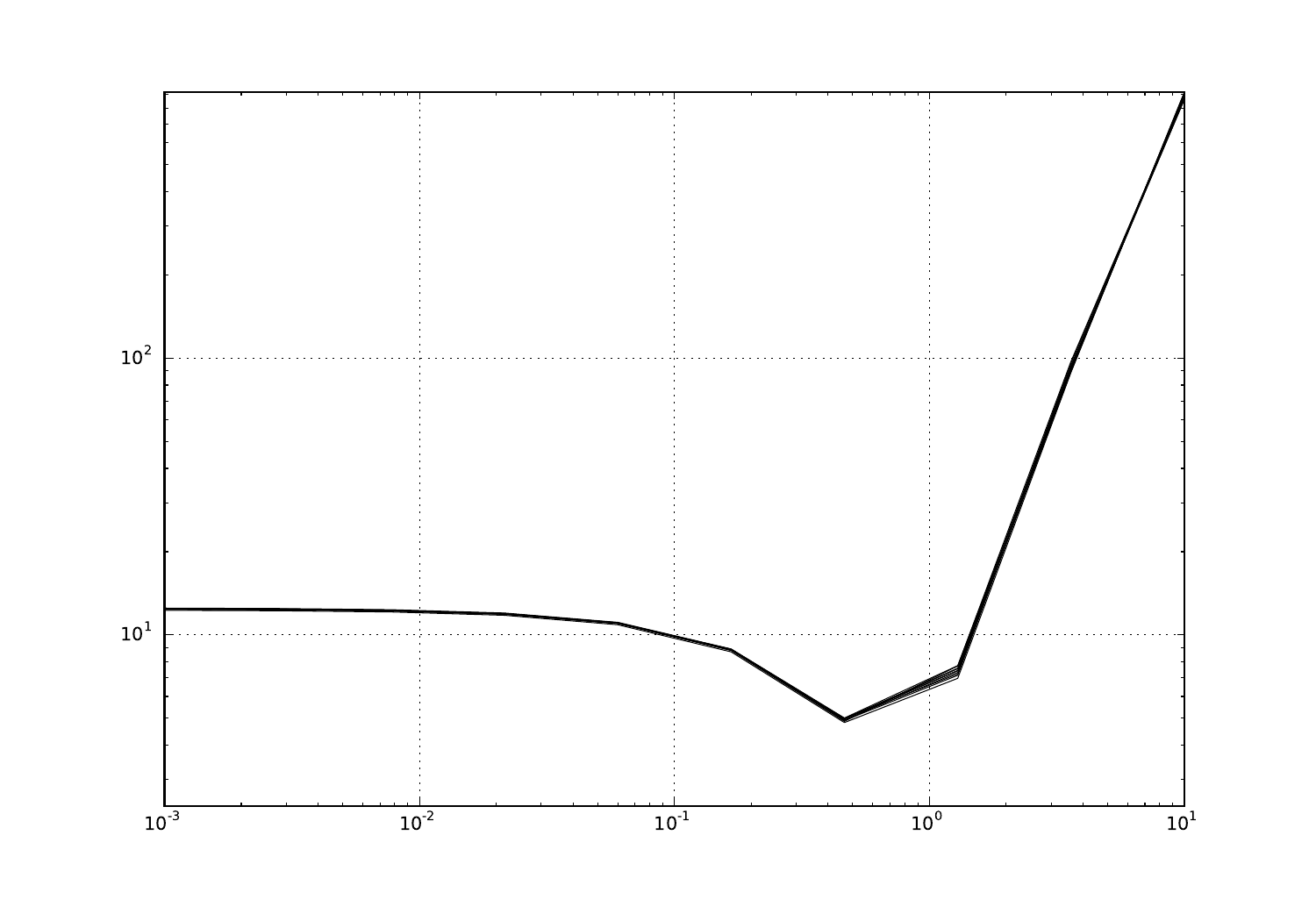}%
	\includegraphics[width=0.34\textwidth]{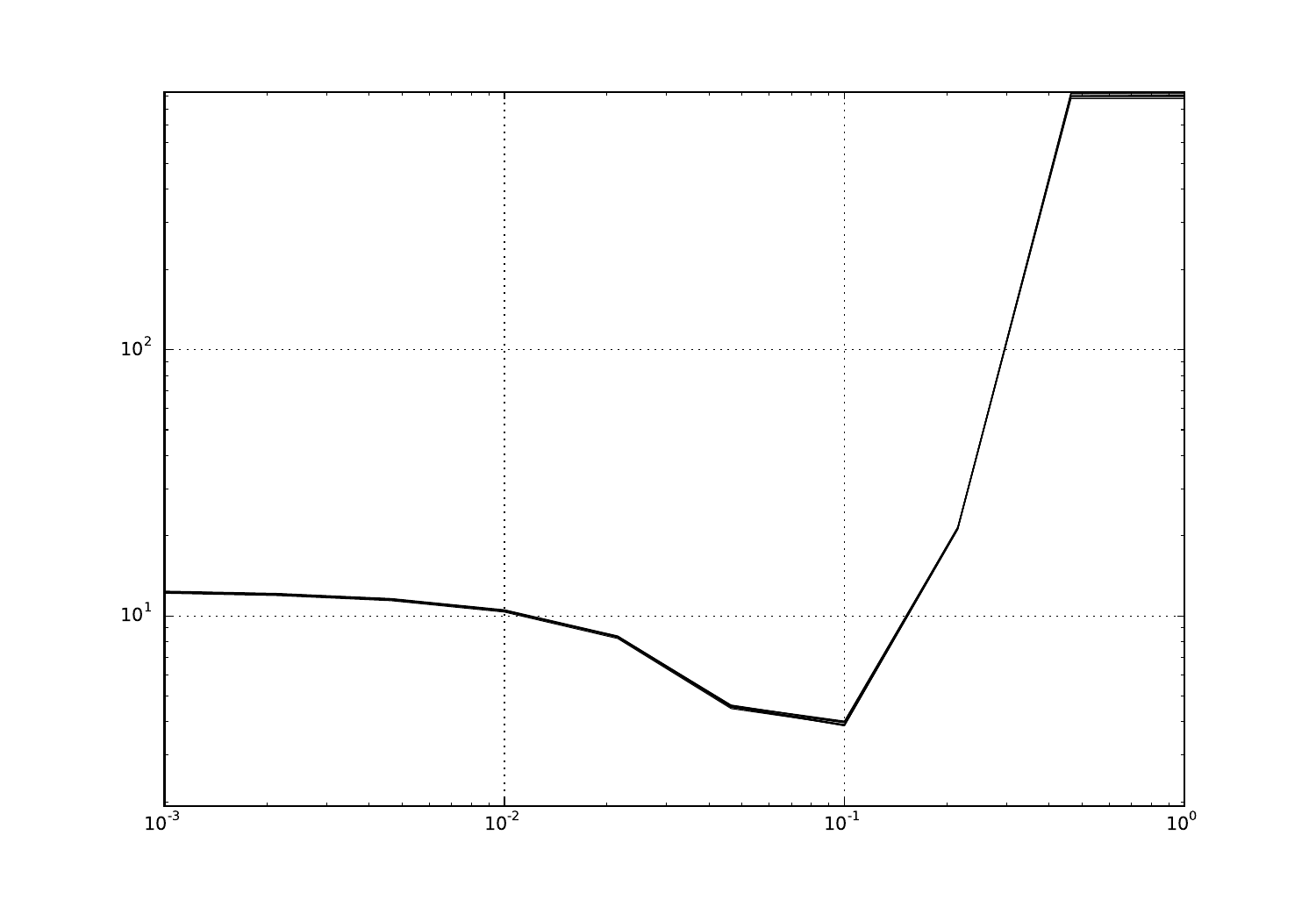}

	\caption{Prediction errors (y-axis) for experiment~2 (see text) for a varying $\lambda$ (x-axis) for procedure~\eqref{eq:procedure1} (first column), procedure~\eqref{eq:procedure2} (second column) and procedure~\eqref{eq:procedure3} (third column). 
	We plot the estimation errors over 10 simulated datasets (corresponding to a line in each figure), for an increasing noise level $\sigma=0.1$ (first line), $\sigma=0.5$ (second line), $\sigma=1.0$ (third line), $\sigma=5.0$ (fourth line).
	We can observe that the optimum $\lambda$ for~\eqref{eq:procedure1} increases with $\sigma$ (see the position of minimum along the first column), while it can be kept almost constant for procedures~\eqref{eq:procedure2} and~\eqref{eq:procedure3}}
	\label{fig:expe2}
\end{figure}

\begin{table}[htbp]
\label{tab:pred_errors2}
  \centering
  \footnotesize
  \noindent
  \begin{tabular}{lcccc}
    Noise level $\sigma$ & 0.1 & 0.5 & 1.0 & 5.0 \\ \hline
	Procedure~\eqref{eq:procedure1} & 1.50e-02   & 6.37e-02  & 2.24e-01  & 6.87e+00   \\
	& (7.82e-03)  & (5.59e-03) & (1.42e-02)  & (1.17e-01) \\
	Procedure~\eqref{eq:procedure2} & 2.05e-03   & 5.01e-02  & 2.01e-01  & 4.95e+00  \\
	& (5.37e-05) & (4.93e-04) & (1.79e-03) & (5.63e-02) \\
	Procedure~\eqref{eq:procedure3} & \textbf{1.64e-03} &  \textbf{4.10e-02} & \textbf{1.64e-01}  & \textbf{3.93e+00} \\
	& (2.61e-05) & (4.40e-04) & (2.78e-03) & (5.87e-02)\\
  \end{tabular}
  \caption{Average best prediction error (and standard deviation) for experiment~2 of the considered procedures for several values $\sigma$. Procedure~\eqref{eq:procedure3} introduced in this paper leads to a strong improvement in this case.}
  
\end{table}

\appendix
\section{Proof of Theorem \ref{thm1}}\label{proof_thm_1}
The proof of Theorem \ref{thm1} is based on the ideas of the proof of Theorem 1 in \cite{Koltchinskii-Tsybakov}. However, as the statistical structure of our estimator is different from that of the estimator proposed in \cite{Koltchinskii-Tsybakov}, the proof requires several modifications and additional information on the behaviour of the estimator. This information is given in Lemmas \ref{l1} and \ref{l2}. In particular, Lemma \ref{l1} provides a bound on the rank of our estimator. Its proof is given in Appendix \ref{proof_lemma_1}
\begin{lemma}\label{l1}
\begin{equation*}
\rank(\hat A)\leq 1/\lambda^{2}.
\end{equation*}
\end{lemma}
\begin{lemma}\label{l2}
Suppose that  $\dfrac{\rho}{\sqrt{\rank(A_0)}}\geq \lambda \geq 3\Delta$ for some $\rho<1$, then
\begin{equation}\label{2}
\Vert  \hat A-\mathbf X\Vert_2\;\geq \left (\dfrac{3-\sqrt{1+\rho^{2}}}{3+\sqrt{1+\rho^{2}}}\right ) \Vert   A_0-\mathbf X\Vert_2
\end{equation}
\end{lemma}
If $\hat A=\mathbf X$, then \eqref{2} implies that $A_0=\mathbf X$ and we get  $\Vert \hat A -A_0\Vert _2=0$.\\
 When $\hat A\neq\mathbf X$, we will use the fact that the subdifferential of the convex function $A\rightarrow  \Vert A \Vert_1$ is the following set of matrices (cf. \cite{watson})
 \begin{equation}\label{subdiff}
 \partial \Vert A \Vert_1=\left\{ \underset{j=1}{\overset{\rank(A)}{\sum}}u_j(A)v_j^{T}(A)+ \mathcal P_{S_1^{\bot}(A)}W\mathcal P_{S_2^{\bot}(A)}\;:\; \Vert  W\Vert_\infty\leq 1\right\}.
 \end{equation} 
  Here $u_j(A)$ and $v_j(A)$ are respectively the left and right orthonormal singular vectors of $A$, $S_1(A)$ is the linear span of $\{u_j(A)\}$, $S_2(A)$ is the linear span of $\{v_j(A)\}$. For simplicity we will write $u_j$ and $v_j$ instead of $u_j(A)$ and $v_j(A)$.
A necessary condition of extremum in \eqref{estimator} implies that there exists $\hat V\in \partial \Vert  \hat A\Vert_1$ such that for any $A\in \mathbb R^{m_1\times m_2}$
\begin{equation}\label{4}
\begin{split}
\dfrac {2\langle \hat A-\mathbf X,\hat A-A\rangle}{2\Vert \hat A-\mathbf X\Vert_2}+\lambda\langle\hat V,\hat A-A\rangle\leq 0.
\end{split}
\end{equation}

By the monotonicity of subdifferentials of convex functions we have that $\langle\hat V-V,\hat A-A\rangle\geq 0$ where $V\in \partial \Vert  A\Vert_1$. Then 
\eqref{4} and $2\langle \hat A-A_0,\hat A-A\rangle= \Vert \hat A-A_0  \Vert^{2}_2+\Vert \hat A-A  \Vert^{2}_2-\Vert A-A_0  \Vert^{2}_2$ imply
\begin{equation}\label{5}
\begin{split}
&\Vert \hat A-A_0  \Vert^{2}_2+\Vert \hat A-A  \Vert^{2}_2+2\lambda\Vert \hat A-\mathbf X\Vert_2\left \langle \mathcal P_{S_1^{\bot}(A)}W\mathcal P_{S_2^{\bot}(A)},\hat A-A\right \rangle\\
&\hskip 0.5 cm\leq\Vert A-A_0  \Vert^{2}_2+2\langle \mathbf X-A_0,\hat A-A\rangle-2\lambda\Vert \hat A-\mathbf X\Vert_2\left \langle \underset{j=1}{\overset{r}{\sum}}u_jv_j^{T},\hat A-A\right\rangle.
\end{split}
\end{equation}
 
 For $B$, a $m_1\times m_2$ matrix,  let $\mathbf {Pr}_A(B)=B- \mathcal P_{S_1^{\bot}(A)}B\mathcal P_{S_2^{\bot}(A)}$. Since 
\begin{equation*}
\mathbf {Pr}_A(B)=\mathcal P_{S_1^{\bot}(A)}B\mathcal P_{S_2(A)}+ \mathcal P_{S_1(A)}B
\end{equation*} 
and $\rank(\mathcal P_{S_i(A)}B)\leq \rank (A)$ we have that $\rank(\mathbf {Pr}_A(B))\leq 2\rank(A)$.
 
Now, we consider each term in \eqref{5} separately. First, using the trace duality and triangle inequality, we get
\begin{equation}\label{7}
\begin{split}
\langle \mathbf X-A_0,\hat A-A\rangle&\leq  \Vert  \mathbf{X}-A_0 \Vert_\infty \Vert \hat A-A \Vert_1\\&\leq
\Vert  \mathbf{X}-A_0 \Vert_\infty \left \Vert\mathbf {Pr}_A\left (\hat A-A \right )\right \Vert_1\\&\hskip 0.5 cm+\Vert  \mathbf{X}-A_0 \Vert_\infty\left  \Vert \mathcal P_{S_1^{\bot}(A)}\left (\hat A-A \right )\mathcal P_{S_2^{\bot}(A)}\right \Vert_1.
\end{split}
\end{equation}
Note that $ \left \Vert \underset{j=1}{\overset{r}{\sum}}u_jv_j^{T}\right  \Vert_\infty=1$. Then, the trace duality implies
\begin{equation}\label{8}
\left \langle \underset{j=1}{\overset{r}{\sum}}u_jv_j^{T},\hat A-A\right \rangle=\left \langle \underset{j=1}{\overset{r}{\sum}}u_jv_j^{T},\mathbf {Pr}_A\left (\hat A-A\right )\right \rangle\leq \left \Vert\mathbf {Pr}_A\left (\hat A-A \right )\right \Vert_1.
\end{equation}
From the trace duality, we get that, there
 exists $W$ with $ \Vert  W\Vert_\infty\leq 1$ such that
 \begin{equation}\label{6}
 \begin{split}
 \left \langle \mathcal P_{S_1^{\bot}(A)}W\mathcal P_{S_2^{\bot}(A)},\hat A-A\right \rangle&=\left\langle W,\mathcal P_{S_1^{\bot}(A)}\left (\hat A-A\right )\mathcal P_{S_2^{\bot}(A)}\right\rangle\\&= \left\Vert \mathcal P_{S_1^{\bot}(A)}\left (\hat A-A\right )\mathcal P_{S_2^{\bot}(A)} \right \Vert_1.
 \end{split}\end{equation}
Using \eqref{2} and the definition of $\lambda$ we derive
\begin{equation}\label{10}
\begin{split}
\lambda\Vert \hat A-\mathbf X\Vert_2\left \Vert \mathcal P_{S_1^{\bot}(A)}\;\hat A\;\mathcal P_{S_2^{\bot}(A)}\right  \Vert_1&\geq \lambda\dfrac{3-\sqrt{1+\rho^{2}}}{3+\sqrt{1+\rho^{2}}} \Vert   A_0-\mathbf X\Vert_2\left \Vert \mathcal P_{S_1^{\bot}(A)}\;\hat A\;\mathcal P_{S_2^{\bot}(A)} \right \Vert_1\\&\geq 3\dfrac{3-\sqrt{1+\rho^{2}}}{3+\sqrt{1+\rho^{2}}}\Vert   A_0-\mathbf X\Vert_\infty \Vert \mathcal P_{S_1^{\bot}(A)}\;\hat A\;\mathcal P_{S_2^{\bot}(A)} \Vert_1.
\end{split}
\end{equation}
Note that $6\dfrac{3-\sqrt{1+\rho^{2}}}{3+\sqrt{1+\rho^{2}}}\geq 2$ for any $\rho<1$. Thus, putting  \eqref{7}, \eqref{8} and \eqref{10} into \eqref{5} yield
\begin{equation}\label{11}
\begin{split}
\Vert \hat A-A_0  \Vert^{2}_2+\Vert \hat A-A  \Vert^{2}_2
&\leq\Vert A-A_0  \Vert^{2}_2+2\Vert  \mathbf{X}-A_0 \Vert_\infty \Vert\mathbf {Pr}_A\left (\hat A-A \right )\Vert_1
\\&\hskip 0.5 cm+2\lambda\Vert \hat A-\mathbf X\Vert_2\Vert\mathbf {Pr}_A\left (\hat A-A \right )\Vert_1.
\end{split}
\end{equation}
Now, using the triangle inequality and the fact that  $$\left \Vert\mathbf {Pr}_A\left (\hat A-A \right )\right \Vert_1\leq \sqrt{2\rank(A)}\Vert \hat A-A \Vert_2$$ we get 
\begin{equation}\label{12}
\begin{split}
&2\left \Vert  \mathbf{X}-A_0 \right \Vert_\infty \left \Vert\mathbf {Pr}_A\left (\hat A-A \right )\right \Vert_1
+2\lambda\left \Vert \hat A-\mathbf X\right \Vert_2\left \Vert\mathbf {Pr}_A\left (\hat A-A \right )\right \Vert_1\\&\hskip 0.5 cm\leq 2\Big (\left \Vert  \mathbf{X}-A_0 \right \Vert_\infty+\lambda\left \Vert  \mathbf{X}-A_0 \right \Vert_2\Big ) \sqrt{2\rank(A)}\left \Vert \hat A-A\right  \Vert_2
\\&\hskip 2.5 cm+2\lambda\Vert \hat A-A_0\Vert_2\sqrt{2\rank(A)}\Vert \hat A-A \Vert_2.
\end{split} 
\end{equation}
From the definition of $\lambda$ we get that $\left \Vert  \mathbf{X}-A_0 \right \Vert_\infty\leq \lambda\left \Vert  \mathbf{X}-A_0 \right \Vert_2/3$. For $A$ such that $\lambda \sqrt{2\rank(A)}\leq \rho$, \eqref{12} implies
\begin{equation*}
\begin{split}
\Vert \hat A-A_0  \Vert^{2}_2+\Vert \hat A-A  \Vert^{2}_2
&\leq\Vert A-A_0  \Vert^{2}_2+\dfrac{8}{3}\lambda\Vert  \mathbf{X}-A_0 \Vert_2\sqrt{2\rank(A)}\Vert \hat A-A \Vert_2
\\&\hskip 1 cm+2\rho\Vert \hat A-A_0\Vert_2\Vert \hat A-A \Vert_2.
\end{split} 
\end{equation*}
Using $2ab\leq a^{2}+b^{2}$ twice we finally compute 
\begin{equation*}
\begin{split}
(1-\rho)\Vert \hat A-A_0  \Vert^{2}_2
\leq\Vert A-A_0  \Vert^{2}_2+\dfrac{4}{1-\rho}\lambda^{2}\Vert  \mathbf{X}-A_0 \Vert_2^{2}\rank(A)
\end{split} 
\end{equation*}
which implies the statement of Theorem \ref{thm1}.

 \section{Proof of Lemma \ref{l1}}\label{proof_lemma_1}
 That $\hat A$ is the minimum of \eqref{estimator} implies that $0\in \partial F(\hat A)$.
 For $\hat A\neq \mathbf X$, \eqref{subdiff} implies that there exists a matrix $W$ such that $\Vert  W\Vert_\infty\leq 1$ and 
  \begin{equation}\label{1}
  \dfrac{\hat A-\mathbf X}{ \Vert  \hat A-\mathbf X\Vert_2}=-\lambda \underset{j=1}{\overset{\rank(\hat A)}{\sum}}u_j(\hat A)v_j^{T}(\hat A)-\lambda \mathcal  P_{S_1^{\bot} (\hat A)}W\mathcal P_{S_2^{\bot}(\hat A)}.
  \end{equation}
 Calculating the $ \Vert \cdot \Vert_2^{2}$ norm of both sides of \eqref{1} we get that $1\geq \lambda^{2}\rank(\hat A)$.

 When $\hat A = \mathbf X$, instead of the differential of $\Vert  \hat A-\mathbf X\Vert_2$ we use its subdifferential: in \eqref{1} the term $ \dfrac{\hat A-\mathbf X}{ \Vert  \hat A-\mathbf X\Vert_2}$ is replaced by  a matrix $\tilde{W}$ such that $ \Vert  \tilde{W}\Vert_2\leq 1$ and we get again $1\geq \lambda^{2}\rank(\hat A)$.

\section{Proof of Lemma \ref{l2}}

If $ A_0= \mathbf X$, then, we have trivially $\Vert  \hat A-\mathbf X\Vert_2\;\geq 0$.

If $ A_0\neq \mathbf X$, by the convexity of the function $A\rightarrow  \Vert   A-\mathbf X\Vert_2$, we have
\begin{equation}\label{3}
\begin{split}
\Vert  \hat A-\mathbf X\Vert_2-\Vert  A_0-\mathbf X\Vert_2&\geq \dfrac{\langle A_0-\mathbf X,\hat A-A_0\rangle}{\Vert  A_0-\mathbf X\Vert_2}\\&\geq -\dfrac{\Vert  A_0-\mathbf X\Vert_\infty}{\Vert   A_0-\mathbf X\Vert_2}\Vert  \hat A-A_0\Vert_1\\
&\geq-\dfrac{\Vert  A_0-\mathbf X\Vert_\infty}{\Vert   A_0-\mathbf X\Vert_2}\sqrt{\rank(\hat A)+\rank(A_0)}\Vert  \hat A-A_0\Vert_2.
\end{split}
\end{equation} 
Using Lemma \ref{l1}, the bound $\dfrac{\rho}{\sqrt{\rank(A_0)}}\geq \lambda$ and the triangle inequality, from \eqref{3} we get
\begin{equation}
\begin{split}
\Vert  \hat A-\mathbf X\Vert_2-\Vert  A_0-\mathbf X\Vert_2
&\geq\\&\hskip -0.5 cm-\dfrac{\sqrt{1+\rho^{2}}}{\lambda}\dfrac{\Vert  A_0-\mathbf X\Vert_\infty}{\Vert   A_0-\mathbf X\Vert_2}\left (\Vert  \hat A-\mathbf X\Vert_2+\Vert  A_0-\mathbf X\Vert_2\right ).
\end{split}
\end{equation} 
Note that $\dfrac{\Vert  A_0-\mathbf X\Vert_\infty}{\lambda\Vert   A_0-\mathbf X\Vert_2}\leq 1/3$ which finally leads to
\begin{equation*}
\left (1+\dfrac{\sqrt{1+\rho^{2}}}{3}\right ) \Vert  \hat A-\mathbf X\Vert_2\;\geq \left (1-\dfrac{\sqrt{1+\rho^{2}}}{3}\right ) \Vert   A_0-\mathbf X\Vert_2.
\end{equation*}
This completes the proof of Lemma \ref{l2}.
\section{Proof of Lemma \ref{l3}}\label{proof_lemma3}
Our goal is to get a numerical estimation on $c_*$ in the case of Gaussian noise. 
Let $Z_i=\xi_i\left (X_i-\bE X_i\right )$ and 
\begin{equation*}
\sigma_Z=\max\left \{\left \Vert \dfrac{1}{n}\Sum \bE\left (Z_iZ^{T}_i\right )\right \Vert^{1/2}_\infty, \left \Vert \dfrac{1}{n}\Sum \bE\left (Z_i^{^{T}}Z_i\right )\right \Vert^{1/2}_\infty\right \}=\dfrac{1}{m_1\wedge m_2}.
\end{equation*}
The constant $c_*$ comes up in the proof of Lemma 2 in \cite{Koltchinskii-Tsybakov} in the estimation of 
\begin{equation*}
\Delta_1=\left \Vert \dfrac{1}{n}\Sum \xi_iX_i\right \Vert_\infty\leq \left \Vert \dfrac{1}{n}\Sum \xi_i\left (X_i-\bE X_i\right )\right \Vert_\infty +\dfrac{1}{\sqrt{m_1m_2}}\left \vert \dfrac{1}{n}\Sum \xi_i\right \vert.
\end{equation*}
A standard application of Markov's inequality gives that, with probability at least $1-1/m$
\begin{equation}\label{B1}
\dfrac{1}{\sqrt{m_1m_2}}\left \vert \dfrac{1}{n}\Sum \xi_i\right \vert\leq 2\sqrt{\dfrac{\log m}{nm_1m_2}}.
\end{equation}
In \cite{Koltchinskii-Tsybakov}, the authors estimate $\left \Vert \dfrac{1}{n}\Sum \xi_i\left (X_i-\bE X_i\right )\right \Vert_\infty$ using \cite [Proposition 2]{koltchinskii-von}. To get a numerical estimation on $c_*$ we follow the lines of the proof of \cite [Proposition 2]{koltchinskii-von}. In order to simplify notations, we write $\Vert\;\Vert_\infty=\Vert\;\Vert$ and we consider the case of Hermitian matrices of size $m'$. Its extension to rectangular matrices is straightforward via self-adjoint dilation, cf., for example, 2.6 in \cite{tropp-user}.

Let $Y_n=\Sum Z_i$. In the proof of \cite [Proposition 2]{koltchinskii-von}, after following the standard derivation of the classical Bernstein inequality and using the Golden-Thompson inequality, the author derives the following bound
\begin{equation}\label{B0}
\mathbb{P}\left (\Vert Y_n\Vert\geq t\right )\leq 2m'e^{-\lambda t}\Vert \bE e^{\lambda Z_1}\Vert^{n}
\end{equation}
and
\begin{equation}\label{B2}
\Vert \bE e^{\lambda Z_1}\Vert\leq 1+\lambda^{2}\left \Vert \bE Z_1^{2}\left [\dfrac{e^{\lambda\Vert Z_1\Vert}-1-\lambda\Vert Z_1\Vert}{\lambda^{2}\Vert Z_1\Vert^{2}}\right ]\right \Vert.
\end{equation}
Using that $\Vert Z_1\Vert\leq 2\vert \xi_i\vert$, from \eqref{B2}, we compute 
\begin{equation}\label{B3}
\begin{split}
\left \Vert \bE e^{\lambda Z_1}\right \Vert&\leq 1+\lambda^{2}\left \Vert \bE\left [\left (X_i-\bE X_i\right )^{2}\right ]\bE \left (\xi_i^{2}\left [\dfrac{e^{2\lambda\vert \xi_i\vert}-1-2\lambda\vert \xi_i\vert}{4\lambda^{2}\xi_i^{2}}\right ]\right )\right \Vert\\&\hskip 0.5 cm \leq 1+\lambda^{2}\sigma_Z^{2}\bE\left (\dfrac{(2\vert \xi_i\vert)^{2}}{2!}+\dfrac{\lambda(2\vert \xi_i\vert)^{3}}{3!}+\cdots\right ).
\end{split}
\end{equation}
Assume that $\lambda <1$, then \eqref{B3} implies
\begin{equation*}
\begin{split}
\left \Vert \bE e^{\lambda Z_1}\right \Vert\leq 1+\lambda^{2}\sigma_Z^{2}\bE e^{2\vert \xi_i\vert}\leq 1+2\lambda^{2}\sigma_Z^{2}e^{2}\leq \exp\{2\lambda^{2}\sigma_Z^{2}e^{2}\}.
\end{split}
\end{equation*}
Using this bound, from \eqref{B0} we get
\begin{equation*}
\mathbb{P}\left (\Vert Y_n\Vert\geq t\right )\leq 2m'\exp\{-\lambda t+2\lambda^{2}\sigma_Z^{2}e^{2}\}.
\end{equation*}
It remains now to minimize the last bound with respect to $\lambda\in(0,1)$ to obtain that
\begin{equation*}
\mathbb{P}\left (\Vert Y_n\Vert\geq t\right )\leq 2m'\exp\left \{-\dfrac{t^{2}}{4e^{2}\sigma_Z^{2}n}\right \}
\end{equation*}
where we supposed that $n$ is large enough. 

Putting $2m'\exp\left \{-\dfrac{t^{2}}{4\sigma_Z^{2}e^{2}n}\right \}=1/(2m')$, we get $t=2e\sqrt{\dfrac{2\log(2m')n}{m_1\wedge m_2}}$. Using \eqref{B1} we compute the following bound on $c_*$
\begin{equation*}
c_*\leq 2e+1\leq 6.5.
\end{equation*}
This completes the proof of Lemma \ref{l3}. 
\section{Proof of Lemma \ref{Lemma}}\label{proof_Lemma}
 Let $\epsilon_i=\sigma \xi_i$. To prove (i) we compute
 \begin{equation}\label{14}
 \begin{split}
 \left\langle\mathbf M,\mathbf M\right\rangle&=\dfrac{ \left\Vert A_0\right\Vert_2^{2}}{\left (m_1m_2\right )^{2}}+\left (1-\dfrac{2n}{m_1m_2}\right )\underset{\mathbf {I}}{\underbrace{\dfrac{1}{n^{2}}\Sum \left\langle A_0,X_i\right\rangle^{2}}} + \underset{\mathbf {II}}{\underbrace{\dfrac{1}{n^{2}}\Sum\epsilon_i^{2}}}\\&\hskip 0.5 cm+\left (1-\dfrac{n}{m_1m_2}\right )\underset{\mathbf{ III}}{\underbrace{\dfrac{2}{n^{2}}\Sum\left\langle A_0,X_i\right\rangle\epsilon_i}}+\underset{\mathbf {IV}}{\underbrace{\dfrac{4}{n^{2}}\underset{i<j}{\sum}\epsilon_i\left\langle A_0,X_j\right\rangle\left\langle X_i,X_j\right\rangle}}\\&\hskip 1 cm+\underset{\mathbf V}{\underbrace{\dfrac{2}{n^{2}}\underset{i<j}{\sum}\epsilon_i \epsilon_j\left\langle X_i,X_j\right\rangle}}+\underset{\mathbf {VI}}{\underbrace{\dfrac{1}{n^{2}}\underset{i\neq j}{\sum} \left\langle A_0,X_i\right\rangle\left\langle A_0,X_j\right\rangle\left\langle X_j,X_i\right\rangle}}.
 \end{split}\end{equation}
 
We estimate each term in \eqref{14} separately with a good probability. 
\begin{itemize}
\item [$\mathbf I$ :] We have that $\bE\left (\dfrac{1}{n^{2}}\Sum \left\langle A_0,X_i\right\rangle^{2}\right )=\dfrac{\left\Vert A_0\right\Vert^{2}_2}{nm_1m_2} $ and $\left |\left\langle A_0,X_i\right\rangle\right |\leq a$. \\Using Hoeffding's inequality 
, we get that, with probability at least \\$1-2\exp\left \{-2\sigma^{4}n/(8a)^{2}\right \}$
$$\dfrac{\left\Vert A_0\right\Vert^2_2}{nm_1m_2} +\dfrac{\sigma^2}{8\,n}\geq\dfrac{1}{n^{2}}\Sum \left\langle A_0,X_i\right\rangle^{2}\geq \dfrac{\left\Vert A_0\right\Vert^2_2}{nm_1m_2} -\dfrac{\sigma^2}{8\,n}.$$

\item [$\mathbf{II}$:] $\epsilon_i^{2}$ are sub-exponential random variables and $\bE\left (\dfrac{1}{n^{2}}\Sum\epsilon_i^{2}\right )=\dfrac{\sigma^{2}}{n}$. Using Bernstein inequality for sub-exponentials random variables (cf. \cite[Proposition 16]{vershynin} ) we get that, with probability at least \\$1-2\exp\left \{-cn\min\left [\sigma^{2}K/8^{2},\sigma\sqrt{K}/8\right ]\right \}$ 
$$\dfrac{\sigma^{2}}{n}+\dfrac{\sigma^{2}}{8n}\geq\dfrac{1}{n^{2}}\Sum\epsilon_i^{2}\geq \dfrac{\sigma^{2}}{n}-\dfrac{\sigma^{2}}{8n}.$$

\item [$\mathbf{III}$:] We have that $\bE\left (\dfrac{2}{n^{2}}\Sum\left\langle A_0,X_i\right\rangle\epsilon_i\right )=0$, using Hoeffding's type inequality for sub-Gaussian random variables (cf. \cite[Proposition 10]{vershynin}) we get
that, with probability at least $1-e\exp\left \{-c\sigma^{2}Kn/a^{2}\right \}$ 
\begin{equation*}
\dfrac{\sigma^{2}}{8n}\geq\dfrac{2}{n^{2}}\Sum\left\langle A_0,X_i\right\rangle\epsilon_i\geq -\dfrac{\sigma^{2}}{8n}.
\end{equation*}

\item [$\mathbf{IV}$:] We compute $\bE\left (\dfrac{4}{n^{2}}\underset{i<j}{\sum}\epsilon_i\left\langle A_0,X_j\right\rangle\left\langle X_i,X_j\right\rangle\right )=0$. We use the following lemma which is proven in the Appendix \ref{proof_l4}.

\begin{lemma}\label{l4}
Suppose that $n\leq m_1m_2$. With probability at least\\ $1-\dfrac{2}{m_1m_2}$  $$\underset{i<j}{\sum}\left\langle X_i,X_j\right\rangle\leq n.$$
\end{lemma} 
Lemma \ref{l4} and Hoeffding's type inequality imply that, with probability at least $1-2/m_1m_2 -e\exp\left \{-c\sigma^{2}nK/a^{2}\right \}$ 
\begin{equation*}
\dfrac{\sigma^{2}}{8n}\geq\dfrac{4}{n^{2}}\underset{i<j}{\sum}\epsilon_i\left\langle A_0,X_j\right\rangle\left\langle X_i,X_j\right\rangle\geq- \dfrac{\sigma^{2}}{8n}
\end{equation*} 

\item [$\mathbf{V}$:] We have that $\bE\left (\dfrac{2}{n^{2}}\underset{i<j}{\sum}\epsilon_i \epsilon_j\left\langle X_i,X_j\right\rangle\right )=0$. Using Bernstein inequality for sub-exponentials random variables (cf. \cite[Proposition 16]{vershynin} ) and Lemma \ref{l4} we get that, with probability at least $1-2\exp\left \{-cn\min\left [\sigma^{2}K/8^{2},\sigma\sqrt{K}/8\right ]\right \}$ 
\begin{equation*}
\dfrac{\sigma^{2}}{8n}\geq\dfrac{2}{n^{2}}\underset{i<j}{\sum}\epsilon_i \epsilon_j\left\langle X_i,X_j\right\rangle\geq -\dfrac{\sigma^{2}}{8n}.
\end{equation*}  

\item [$\mathbf{VI}$:]  We compute that 
\begin{equation*}
\begin{split}
\bE\left (\dfrac{1}{n^{2}}\underset{i\neq j}{\sum} \left\langle A_0,X_i\right\rangle\left\langle A_0,X_j\right\rangle\left\langle X_j,X_i\right\rangle\right )&=\dfrac{1}{n^{2}}\underset{i\neq j}{\sum} \left\langle \bE\left (\left\langle A_0,X_j\right\rangle X_j\right ),\bE \left (\left\langle A_0,X_i\right\rangle X_i\right )\right\rangle 
 \\&=\dfrac{1}{n^{2}}\underset{i\neq j}{\sum} \dfrac{\left\Vert A_0\right\Vert^{2}_2}{\left  (m_1m_2\right  )^{2}} \\&\leq\dfrac{\left\Vert A_0\right\Vert^{2}_2}{\left  (m_1m_2\right  )^{2}}.
\end{split}
\end{equation*}
 Using Lemma \ref{l4} and Hoeffding's type inequality for sub-Gaussian random variables (cf. \cite[Proposition 10]{vershynin}), we get that, with probability at least $1-2/m_1m_2-2\exp\left \{-2\sigma^{4}n/(8a)^{2}\right \}$
\begin{equation*}
\dfrac{1}{n^{2}}\underset{i\neq j}{\sum} \left\langle A_0,X_i\right\rangle\left\langle A_0,X_j\right\rangle\left\langle X_j,X_i\right\rangle\leq \dfrac{\left\Vert A_0\right\Vert^{2}_2}{(m_1m_2)^{2}}+\dfrac{\sigma^{2}}{8n}.
\end{equation*}

\end{itemize}
To obtain the lower bound, note that, for $i\neq j$, $\left\langle X_i,X_j\right\rangle\neq 0$ iff $X_i=X_j$. This implies that $\underset{i\neq j}{\sum} \left\langle A_0,X_i\right\rangle\left\langle A_0,X_j\right\rangle\left\langle X_j,X_i\right\rangle\geq 0$. We use that $2n<m_1m_2$ to get $$\dfrac{ \left\Vert A_0\right\Vert_2^{2}}{\left (m_1m_2\right )^{2}}+\left (1-\dfrac{2n}{m_1m_2}\right )\dfrac{1}{n^{2}}\Sum \left\langle A_0,X_i\right\rangle^{2}\geq 0.$$ Putting the lower bounds in $\mathbf{II}-\mathbf{V}$ together we compute from \eqref{14}
  \begin{equation*}
   \left\Vert\mathbf M\right\Vert^{2}_2\geq
  \dfrac{\sigma^2}{2\,n}. 
  \end{equation*}
 
 To obtain the upper bound, we use the upper bounds in $\mathbf{I}-\mathbf{VI}$. From \eqref{14} we get
 \begin{equation*}
    \left\Vert\mathbf M\right\Vert^{2}_2\leq \dfrac{2\left\Vert A_0\right\Vert^2_2}{\left (m_1m_2\right )^{2}}+
   \dfrac{\left\Vert A_0\right\Vert^2_2}{nm_1m_2} +\dfrac{14\sigma^2}{8\,n}\leq 2\left  (\dfrac{\left\Vert A_0\right\Vert^2_2}{nm_1m_2} +\dfrac{\sigma^2}{\,n}\right )
   \end{equation*}
   where we used that $2n\leq m_1m_2$. This completes the proof of part (i) in Lemma \ref{Lemma}.
   
   To prove (ii) we use
     that $\left\langle X_i,X_i\right\rangle=1$ and $\left\langle X_i,X_j\right\rangle\neq 0$  iff $X_i=X_j$. We compute
   \begin{equation*}
   \begin{split}
   \dfrac{1}{n^{2}} \left\langle\Sum Y_i X_i,\Sum Y_i X_i\right\rangle&=\dfrac{1}{n^{2}}\Sum Y_i^{2}+\dfrac{2}{n^{2}}\underset{i<j}{\sum}Y_i Y_j\left\langle X_i,X_j\right\rangle\\&=\dfrac{1}{n^{2}}\Sum \left (\left\langle A_0,X_i\right\rangle^{2}+\epsilon_i^{2}+2\left\langle A_0,X_i\right\rangle\epsilon_i\right )\\&\hskip 0.5 cm+\dfrac{2}{n^{2}}\underset{i<j}{\sum}\left\langle A_0,X_i\right\rangle^{2}\left\langle X_i,X_j\right\rangle\\&\hskip 0.75
   cm+\dfrac{4}{n^{2}}\underset{i<j}{\sum}\epsilon_i\left\langle A_0,X_j\right\rangle\left\langle X_i,X_j\right\rangle +\dfrac{2}{n^{2}}\underset{i<j}{\sum}\epsilon_i \epsilon_j\left\langle X_i,X_j\right\rangle.
   \end{split}\end{equation*}
   This implies that
   \begin{equation}\label{15}
   \begin{split}
   \dfrac{1}{n^{2}} \left\langle\Sum Y_i X_i,\Sum Y_i X_i\right\rangle&\geq \underset{\mathbf I}{\underbrace{\dfrac{1}{n^{2}}\Sum \left\langle A_0,X_i\right\rangle^{2}}}+\underset{\mathbf {II}}{\underbrace{\dfrac{1}{n^{2}}\Sum\epsilon_i^{2}}}+\underset{\mathbf{ III}}{\underbrace{\dfrac{2}{n^{2}}\Sum\left\langle A_0,X_i\right\rangle\epsilon_i}}\\&\hskip 0.5  cm +\underset{\mathbf {IV}}{\underbrace{\dfrac{4}{n^{2}}\underset{i<j}{\sum}\epsilon_i\left\langle A_0,X_j\right\rangle\left\langle X_i,X_j\right\rangle}}+\underset{\mathbf V}{\underbrace{\dfrac{2}{n^{2}}\underset{i<j}{\sum}\epsilon_i \epsilon_j\left\langle X_i,X_j\right\rangle}}.
   \end{split}\end{equation}
   Using the lower bounds for $\mathbf{I} - \mathbf{V}$ we get from \eqref{15}
   \begin{equation*}
   \dfrac{1}{n^{2}} \left\langle\Sum Y_i X_i,\Sum Y_i X_i\right\rangle\geq \dfrac{ \left\Vert A_0\right\Vert_2^{2}}{\,nm_1m_2}
   \end{equation*}
   which proves the part (ii) of Lemma \ref{Lemma}.

   (iii) is a consequence of (ii). For $4n\leq m_1m_2$ (ii) implies 
   \begin{equation*}
      \dfrac{1}{4n^{2}} \left\langle\Sum Y_i X_i,\Sum Y_i X_i\right\rangle
      \geq \dfrac{ \left\Vert A_0\right\Vert_2^{2}}{\left (m_1m_2\right )^{2}}.
      \end{equation*}
  Now we complete the proof of part (iii) of Lemma \ref{Lemma} using that $$\left\Vert\mathbf M\right\Vert_2\geq  \left\Vert \dfrac{1}{n}\Sum Y_iX_i\right\Vert_2-\dfrac{\left\Vert A_0\right\Vert_2}{m_1m_2}.$$
  
\section{Proof of Lemma \ref{l4}}\label{proof_l4}
Recall that for $i\neq j$, $X_i$ and $X_j$ are independent. We compute the expectation
\begin{equation*}
\bE\left (\underset{i<j}{\sum}\left\langle X_i,X_j\right\rangle\right )=\underset{i<j}{\sum}\left\langle \bE X_i,\bE X_j\right\rangle=\dfrac{n(n-1)}{2m_1m_2}
\end{equation*}
and the variance
\begin{equation*}
\begin{split}
\bE\left (\left (\underset{i<j}{\sum}\left\langle X_i,X_j\right\rangle\right )^{2}\right )-\left (\bE\left (\underset{i<j}{\sum}\left\langle X_i,X_j\right\rangle\right )\right) ^{2}&\\&\hskip -6 cm =\bE\left (\underset{\underset{i'<j'}{i<j}}{\sum}\left\langle  X_i, X_j\right\rangle\left\langle  X_{i'}, X_{j'}\right\rangle\right )-
\underset{\underset{i'<j'}{i<j}}{\sum}\bE\left (\left\langle  X_i, X_j\right\rangle\right )\bE\left (\left\langle  X_{i'}, X_{j'}\right\rangle\right ).
\end{split}
\end{equation*}
When $i,j,i',j'$ are all distinct, $\bE\left ( \left\langle  X_i, X_j\right\rangle\left\langle  X_{i'}, X_{j'}\right\rangle\right )$ is cancelled by the corresponding term in $\underset{\underset{i'<j'}{i<j}}{\sum}\bE\left (\left\langle  X_i,X_j\right\rangle\right )\bE\left (\left\langle X_{i'}, X_{j'}\right\rangle\right )$. 

It remains to consider the following five cases: (1) $i=i'$ and $j=j'$; (2) $i=i'$ and $j\neq j'$; (3) $i\neq i'$ and $j= j'$; (4) $i=j'$ and $j\neq i'$; (5) $i'=j$ and $j'\neq i$.

\textbf{Case (1):} note that $\left \langle  X_i, X_j\right\rangle$ takes only two values $0$ or $1$, which implies that
\begin{equation*}
\bE \left (\left \langle  X_i, X_j\right\rangle^{2}\right )=\bE \left (\left \langle  X_i, X_j\right\rangle\right )=\dfrac{1}{m_1m_2}.
\end{equation*} 
\textbf{Cases (2)-(5):} in these four cases, we need to calculate $\bE \left (\left \langle  X_i, X_k\right\rangle\left \langle  X_k, X_j\right\rangle\right )$ for $i\neq j$ and $k\notin \{i,j\}$. Note that $\mathcal{P}_{X_k}=\left \langle  \,\cdot\,, X_k\right\rangle X_k$ is the orthogonal projector on the vector space spanned by $X_k$. We compute $$\bE \mathcal{P}_{X_k}=\dfrac{1}{m_1m_2}\mathrm{Id}$$ where $\mathrm {Id}$ is the identity application on $\mathbb{R}^{m_1\times m_2}$. Then, we get 
\begin{equation*}
\begin{split}
\bE \left (\left \langle \left \langle  X_i, X_k\right\rangle X_k, X_j\right\rangle\right ) &=\bE \left (\left\langle \mathcal{P}_{X_k}\left (X_i\right ), X_j\right\rangle\right )
\\&=\left  \langle \bE\left (\mathcal{P}_{X_k}\right ) \left (\bE X_i\right ),\bE X_j\right\rangle
\\&=\dfrac{1}{m_1m_2}\left\langle \bE X_i,\bE X_j\right\rangle\\&=\dfrac{1}{\left (m_1m_2\right )^{2}}.
\end{split}\end{equation*}
These terms are cancelled by the corresponding terms in $\underset{\underset{i'<j'}{i<j}}{\sum}\bE\left (\left\langle  X_i, X_j\right\rangle\right )\bE\left (\left\langle  X_{i'},X_{j'}\right\rangle\right )$ as $$\bE\left (\left \langle  X_i, X_k\right\rangle\right ) \bE\left (\left \langle  X_k, X_j\right\rangle\right )=\dfrac{1}{(m_1m_2)^{2}}.$$

Finally we get that
\begin{equation*}
\bE\left (\left (\underset{i<j}{\sum}\left\langle X_i,X_j\right\rangle\right )^{2}\right )-\left (\bE\left (\underset{i<j}{\sum}\left\langle X_i,X_j\right\rangle\right )\right) ^{2}\leq \dfrac{n(n-1)}{2m_1m_2}.
\end{equation*}
The Bienaym\'e-Tchebychev inequality implies that
\begin{equation*}
\mathbb{P}\left (\underset{i<j}{\sum}\left\langle X_i,X_j\right\rangle\geq n\right )\leq \dfrac{n(n-1)}{2m_1m_2\left (n-\dfrac{n(n-1)}{2m_1m_2}\right )^{2}}\leq \dfrac{2}{m_1m_2}
\end{equation*}
when $m_1m_2\geq n$. This completes the proof of Lemma \ref{l4}.
\section{Proof of Theorem \ref{thmr1}}\label{proof_thmr1}
The following lemma is the counterpart of Lemma \ref{l1} in the present setting. It is proven in the Appendix \ref{proof_lr1}. 
     \begin{lemma}\label{lr1}
     \begin{equation*}
     \rank(V\hat A)\leq 1/\lambda^{2}.
     \end{equation*}
     \end{lemma}
    
We need the following auxiliary result, which corresponds to Lemma \ref{l2}, and which is proven in the Appendix \ref{proof_lr2}.
     \begin{lemma}\label{lr2}
     Suppose that  $\dfrac{\rho}{\sqrt{\rank(VA_0)}}\geq \lambda \geq 3\Delta'$ for some $\rho<1$, then
     \begin{equation*}
     \left \Vert  V\hat A-U\right \Vert_2\;\geq \left (\dfrac{3-\sqrt{1+\rho^{2}}}{3+\sqrt{1+\rho^{2}}}\right ) \Vert   E\Vert_2.
     \end{equation*}
     \end{lemma}
    Now the proof of Theorem \ref{thmr1} is very similar to the proof of the Theorem \ref{thm1}. We only sketch it. 
     If $V\hat A\neq U$, a necessary condition of extremum in \eqref{estimator_regression} implies that there exists a $\hat W\in \partial \Vert V\hat A \Vert_1$ such that for any $A\in \mathbb R^{m_1\times m_2}$
     \begin{equation}\label{r4}
     \begin{split}
     \dfrac {2\left \langle V\hat A-U,V\left (\hat A-A\right )\right \rangle}{2\left \Vert V\hat A-U\right \Vert_2}+\lambda\left \langle\hat W,V\left (\hat A-A\right )\right \rangle\leq 0.
     \end{split}
     \end{equation}
%
      and we get
     \begin{equation}\label{r5}
     \begin{split}
     \left \Vert V\left (\hat A-A_0 \right )\right  \Vert^{2}_2+\left \Vert V\left (\hat A-A\right )\right   \Vert^{2}_2+2\lambda\Big \Vert V\hat A-U\Big \Vert_2\left \langle \mathcal P_{S_1^{\bot}(VA)}W\mathcal  P_{S_2^{\bot}(VA)},V\left (\hat A-A\right )\right \rangle&\\&\hskip -12 cm
     \leq\big \Vert V\left ( A-A_0 \right )\big  \Vert^{2}_2
     +2\left \langle E,V\left (\hat A-A\right )\right \rangle\\&\hskip -9 cm -2\lambda\left \Vert V\hat A-U\right \Vert_2\left \langle \underset{j=1}{\overset{\rank(VA)}{\sum}}u_j(VA)v_j(VA)^{T},V\left (\hat A-A\right )\right\rangle.
     \end{split}
     \end{equation}

      Let $\mathbf {Pr}_{VA}(B)=B- \mathcal P_{S_1^{\bot}(VA)}B \mathcal P_{S_2^{\bot}(VA)}$. Then,
      the trace duality and the triangle inequality imply
     \begin{equation}\label{r7}
     \begin{split}
     \left \langle E,V\left (\hat A-A\right )\right \rangle&=\left \langle \mathcal{P}_VE,V\left (\hat A-A\right )\right \rangle \\&\leq  \left \Vert  \mathcal{P}_VE\right  \Vert_\infty \left \Vert V\left ( \hat A-A\right )\right   \Vert_1 \\&\leq
     \left \Vert  \mathcal{P}_VE\right  \Vert_\infty \left \Vert\mathbf {Pr}_{VA}\left [ V\left (\hat A-A \right )\right ]\right \Vert_1 \\&\hskip 0.5 cm+\left \Vert  \mathcal{P}_VE\right  \Vert_\infty\left  \Vert\mathcal  P_{S_1^{\bot}(VA)}V\left (\hat A-A \right )\mathcal P_{S_2^{\bot}(VA)}\right \Vert_1.
     \end{split}
     \end{equation}
    Using $6\dfrac{3-\sqrt{1+\rho^{2}}}{3+\sqrt{1+\rho^{2}}}\geq 2$ for any $\rho<1$  \eqref{r5} implies
     \begin{equation}\label{r11}
     \begin{split}
     \left \Vert V\left (\hat A-A_0 \right )\right  \Vert^{2}_2+\left \Vert V\left (\hat A-A\right )\right   \Vert^{2}_2&
                  \leq \big \Vert V\left ( A-A_0 \right )\big  \Vert^{2}_2\\&\hskip 0.5 cm+2\left \Vert  \mathcal{P}_VE\right  \Vert_\infty \left \Vert\mathbf {Pr}_{VA} \left [V\left (\hat A-A \right )\right ]\right \Vert_1 \\&\hskip 1 cm
                    +2\lambda\left \Vert V\hat A-U\right \Vert_2\left \Vert\mathbf {Pr}_{VA} \left [V\left (\hat A-A \right )\right ]\right \Vert_1.
                         \end{split}
     \end{equation}
     Now we use  $\left \Vert\mathbf {Pr}_{VA} \left [V\left (\hat A-A \right )\right ]\right \Vert_1\leq \sqrt{2\rank(VA)}\left \Vert V\left (\hat A-A \right )\right \Vert_2$, $\left \Vert  \mathcal{P}_VE\right  \Vert_\infty\leq \lambda\Vert  E \Vert_2/3$ and $\lambda\sqrt{2\,\rank(VA)}\leq \rho$ to conclude
     \begin{equation*}
     \begin{split}
     (1-\rho)\left \Vert V\left (\hat A-A_0\right )\right   \Vert^{2}_2
     \leq \big\Vert V\left( A-A_0\right )\big\Vert^{2}_2+\dfrac{4\lambda^{2}}{1-\rho}\Vert  E \Vert_2^{2}\,\rank(VA)
     \end{split} 
     \end{equation*}
     which implies the statement of Theorem \ref{thmr1}.
     
\section{Proof of Lemma \ref{lr1}}\label{proof_lr1}
That $\hat A$ is the minimum of \eqref{estimator_regression} implies that $0\in \partial G(\hat A)$ where $$G=\Vert U-V\,A\Vert_2+\lambda \Vert VA\Vert_1.$$ Note that the subdifferential of the convex function $A\rightarrow  \Vert VA \Vert_1$ is the following set of matrices 
     \begin{equation*}
     \partial \Vert VA \Vert_1=V^{T}\left\{ \underset{j=1}{\overset{\rank(VA)}{\sum}}u_j(VA)v_j^{T}(VA)+ \mathcal P_{S_1^{\bot}(VA)}W\mathcal P_{S_2^{\bot}(VA)}\;:\; \Vert  W\Vert_\infty\leq 1\right\}
     \end{equation*} 
      where $S_1(VA)$ is the linear span of $\{u_j(VA)\}$ and $S_2(VA)$ is the linear span of $\{v_j(VA)\}$.
 
 If $\hat A$ is such that $V\hat A\neq U$, we obtain that, there exists a matrix $W$ such that $\Vert  W\Vert_\infty\leq 1$ and 
      \begin{equation*} 
      V^{T}\dfrac{V\hat A-U}{ \Vert V \hat A-U\Vert_2}=-\lambda V^{T}\left\{ \underset{j=1}{\overset{\rank(VA)}{\sum}}u_j(VA)v_j^{T}(VA)+ \mathcal P_{S_1^{\bot}(VA)}W\mathcal P_{S_2^{\bot}(VA)}\right \} 
      \end{equation*}
      which implies
      \begin{equation}\label{r1}
            V^{T}\mathcal{P}_V\dfrac{V\hat A-U}{ \Vert V \hat A-U\Vert_2}=-\lambda V^{T}\mathcal{P}_V\left\{ \underset{j=1}{\overset{\rank(VA)}{\sum}}u_j(VA)v_j^{T}(VA)+\mathcal P_{S_1^{\bot}(VA)}W\mathcal P_{S_2^{\bot}(VA)}\right \}. 
            \end{equation}
Using $\mathcal{P}_V VA\left (v_j(VA)\right )=VA\left (v_j(VA)\right )=\sigma_j(VA)u_j(VA)$ and $\sigma_j\neq 0$ we get
                  \begin{equation}\label{41}
                                    \mathcal{P}_Vu_j(VA)=u_j(VA).
                                    \end{equation}
                                     Note that  for any $w$ such that $\left \langle w,u_j(VA)\right \rangle=0$ \eqref{41} implies that 
                                    \begin{equation}\label{42}
                                    \left \langle \mathcal{P}_Vw,u_j(VA)\right \rangle=\left \langle w,u_j(VA)\right \rangle=0.
                                    \end{equation}                                    
 By the definition, $\mathcal P_{S_1^{\bot}(VA)}$ projects on the orthogonal complement of the linear span of $\{u_j(VA)\}$. Thus, \eqref{42}  implies that  $\mathcal{P}_V\mathcal P_{S_1^{\bot}(VA)}$ also projects on the subspace orthogonal to the linear span of $\{u_j(VA)\}$.
                                                                         
      Note that $V^{T}\mathcal{P}_VB=0$ imply 
      $\mathcal{P}_VB=0$ and we get from \eqref{r1}
                  
     \begin{equation}\label{r3}
                       \mathcal{P}_V\dfrac{V\hat A-U}{ \Vert V \hat A-U\Vert_2}=-\lambda \left\{ \underset{j=1}{\overset{\rank(VA)}{\sum}}u_j(VA)v_j^{T}(VA)+ \mathcal{P}_V\left [\mathcal P_{S_1^{\bot}(VA)}W\mathcal P_{S_2^{\bot}(VA)}\right ]\right \}.
                       \end{equation}
                                   
                       
     Calculating the $ \Vert \; \Vert_2^{2}$ norm of both sides of \eqref{r3} we get that $1\geq \lambda^{2}\rank(V\hat A)$.
     When $V\hat A= U$, instead of the differential of $\Vert U-V\,A\Vert_2$ we use its subdiffential.


\section{Proof of Lemma \ref{lr2}}\label{proof_lr2}
If $V A_0= U$, then we have trivially $ \Vert  V\hat A-U\Vert_2\;\geq 0.$
If $ VA_0\neq U$, by the convexity of function $A\rightarrow  \Vert  V A-U\Vert_2$, we have
\begin{equation}\label{r33}
\begin{split}
\left \Vert  V\hat A-U\right \Vert_2-\left \Vert V A_0-U\right \Vert_2&\geq \dfrac{\left \langle VA_0-U,V\left (\hat A-A_0\right )\right \rangle}{\left \Vert V A_0-U\right \Vert_2}\\&=\dfrac{\left \langle\mathcal{P}_V\left (E\right ) ,V\left (\hat A-A_0\right )\right \rangle}{\left \Vert V A_0-U\right \Vert_2}\\&\geq -\dfrac{\left \Vert  \mathcal{P}_V\left (E\right )\right \Vert_\infty}{\left \Vert  E\right \Vert_2}\left \Vert V\left ( \hat A-A_0\right )\right \Vert_1\\
&\geq-\dfrac{\left \Vert  \mathcal{P}_V\left (E\right )\right \Vert_\infty}{\left \Vert  E\right \Vert_2}\sqrt{\rank(VA_0)+\rank(V\hat A)}\left \Vert  V\left (\hat A-A_0\right )\right \Vert_2.
\end{split}
\end{equation} 
Using the bound $\dfrac{\rho}{\sqrt{\rank(VA)}}\geq \lambda$, Lemma \ref{lr1} and the triangle inequality from \eqref{r33} we get 
\begin{equation*}
\begin{split}
\left \Vert  V\hat A-U\right \Vert_2-\left \Vert  VA_0-U\right \Vert_2
&\geq\\&\hskip -0.5 cm-\dfrac{\sqrt{1+\rho^{2}}}{\lambda}\dfrac{\left \Vert  \mathcal{P}_V\left (E\right )\right \Vert_\infty}{\left \Vert  E\right \Vert_2}\left (\Vert  V\hat A-U\Vert_2+\Vert V A_0-U\Vert_2\right ).
\end{split}
\end{equation*} 
By the definition of $\lambda$ we have $\dfrac{\left \Vert  \mathcal{P}_V\left (E\right )\right \Vert_\infty}{\lambda\left \Vert  E\right \Vert_2}\leq 1/3$ which finally leads to
\begin{equation*}
\left (1+\sqrt{1+\rho^{2}}/3\right ) \Vert  V\hat A-U\Vert_2\;\geq \left (1-\sqrt{1+\rho^{2}}/3\right ) \Vert   VA_0-U\Vert_2.
\end{equation*}
This completes the proof of Lemma \ref{lr2}.
\vskip 0.5 cm
\textbf{Acknowledgements}.
It is a pleasure to thank A. Tsybakov for introducing me this problem and illuminating discussions.

\bibliographystyle{plain}

\bibliography{biblio}

\end{document}